\documentclass[11pt,reqno]{amsart}
\usepackage[tt=false]{libertine}
\usepackage{amssymb}
\usepackage{mathtools}
\usepackage[varbb]{newpxmath}
\usepackage[foot]{amsaddr}

\let\savedbigtimes\bigtimes
\let\bigtimes\relax
\usepackage{mathabx}
\let\bigtimes\savedbigtimes

\usepackage[margin=1in, bottom=1in]{geometry}
\usepackage{graphicx}
\usepackage{caption}
\usepackage{subcaption}
\usepackage{enumitem}
\usepackage{comment}
\usepackage{bbm}

\usepackage[usenames,dvipsnames]{xcolor}
\usepackage[colorlinks=true,
  linkcolor=teal!60!black,
  citecolor=PineGreen,
  urlcolor=RedViolet]{hyperref}
\hypersetup{bookmarksopen=true}

\usepackage[yyyymmdd,hhmmss]{datetime}

\usepackage{tikz}
\usetikzlibrary{decorations.pathreplacing,calligraphy}
\usetikzlibrary{calc}

\newtheorem{thm}{Theorem}[section]
\newtheorem{lem}[thm]{Lemma}
\newtheorem{ppn}[thm]{Proposition}

\newtheorem{cor}[thm]{Corollary}

\theoremstyle{definition}
\newtheorem{dfn}[thm]{Definition}

\theoremstyle{remark}

\numberwithin{equation}{section}
\def\beq#1\eeq{%
    \begin{equation}%
    #1%
    \end{equation}%
}
\def\baln#1\ealn{%
    \begin{align}%
    #1%
    \end{align}%
}
\def\balnn#1\ealnn{%
    \begin{align*}%
    #1%
    \end{align*}%
}

\def\lt{\left}
\def\rt{\right}
\def\fr{\frac}
\def\la{\langle}
\def\ra{\rangle}
\def\eps{\varepsilon}

\def\bI{{\boldsymbol{I}}}

\def\bW{{\boldsymbol{W}}}

\def\bxi{{\boldsymbol{\xi}}}

\def\be{{\boldsymbol{e}}}

\def\bg{{\boldsymbol{g}}}

\def\bm{{\boldsymbol{m}}}

\def\bu{{\boldsymbol{u}}}
\def\bv{{\boldsymbol{v}}}

\def\bx{{\boldsymbol{x}}}
\def\by{{\boldsymbol{y}}}

\def\bbE{{\mathbb{E}}}

\def\bbN{{\mathbb{N}}}

\def\bbP{{\mathbb{P}}}
\def\bbQ{{\mathbb{Q}}}
\def\bbR{{\mathbb{R}}}
\def\bbS{{\mathbb{S}}}

\def\bbW{{\mathbb{W}}}

\def\bbone{{\mathbb{1}}}

\def\cA{{\mathcal{A}}}
\def\cB{{\mathcal{B}}}

\def\cD{{\mathcal{D}}}
\def\cE{{\mathcal{E}}}
\def\cF{{\mathcal{F}}}
\def\cG{{\mathcal{G}}}
\def\cH{{\mathcal{H}}}

\def\cN{{\mathcal{N}}}

\def\cP{{\mathcal{P}}}

\def\cR{{\mathcal{R}}}

\def\cW{{\mathcal{W}}}
\def\cX{{\mathcal{X}}}

\def\sD{{\mathscr{D}}}

\def\Airy{{\mathsf{Airy}}}
\def\de{{\mathsf{d}}}
\def\diag{{\mathsf{diag}}}

\def\GOE{{\mathsf{GOE}}}
\def\Law{{\mathsf{Law}}}

\def\sph{{\mathsf{sph}}}

\def\Tr{{\mathsf{Tr}}}

\def\unif{{\mathsf{unif}}}
\def\Var{{\mathsf{Var}}}

\def\oB{{\bar{\operatorname{B}}}}

\def\fP{{\mathfrak P}}

\title{Overlap distribution of the critical Sherrington--Kirkpatrick model}

\author{Hang Du \and Brice Huang}

\address[H.~Du]{Department of Mathematics, Massachusetts Institute of Technology}
\email[H.~Du]{hangdu@mit.edu}

\address[B.~Huang]{Department of Statistics, Stanford University}
\email[B.~Huang]{bmhuang@stanford.edu}

\date{\today}

\subjclass[2020]{Primary 82B44; secondary 60K35, 60B20, 82B27.}

\begin{document}

\begin{abstract}
  We study the distribution of the two-replica overlap $R_{1,2}$ in the Ising and spherical Sherrington--Kirkpatrick models at the critical inverse temperature $\beta = 1$.
  Our main result shows that in both models, $R_{1,2}$ has scale $N^{-1/3}$, and the quenched distribution of $N^{1/3} R_{1,2}$ converges to an explicit random probability measure defined in terms of the reflected $\Airy_1$ point process.
  As a consequence, we characterize the limiting value of $N^{2/3} \bbE \la R_{1,2}^2 \ra$, answering a question of Talagrand \cite{talagrand2011mean2}.

  For the spherical SK model, we obtain the limit by representing the Gibbs measure as an anisotropic Gaussian on $\bbR^N$ conditioned to have norm $\sqrt{N}$, and then passing to the $\Airy_1$ scaling limit at the GOE spectral edge.
  For the Ising SK model, the proof is based on a sphere-to-cube comparison principle showing that the quenched distributions of $N^{1/3} R_{1,2}$ under the spherical and Ising Gibbs measures asymptotically coincide.

  This paper is a companion to \cite{du2026fluctuations}, where we introduced a related comparison principle to identify the limiting fluctuations of the SK free energy.
  Most of the arguments in this paper were generated using GPT-5.6 Pro, with the aim of exploring further consequences of the ideas developed in that work.
\end{abstract}

\maketitle

\setcounter{tocdepth}{1}
\tableofcontents

\section{Introduction and main results}
\label{s:intro}

Let $N$ be a positive integer and $\bW \sim \GOE(N)$ be an $N\times N$ symmetric matrix whose upper-triangular entries are independent, with $W_{i,i} \sim \cN(0, 2/N)$ and $W_{i,j} \sim \cN(0, 1/N)$ for $i<j$.
The Sherrington--Kirkpatrick (SK) model \cite{sherrington1975solvable} has Hamiltonian $H_N : \bbR^N \rightarrow \bbR$ given by
\[
  H_N(\bx) = \fr12 (\bW \bx, \bx)\,.
\]
This is equivalently the Gaussian process on $\bbR^N$ with covariance
\[
  \bbE H_N(\bx)H_N(\by) = \fr{N}{2} R(\bx,\by)^2\,,
\]
where $R(\bx,\by) = (\bx,\by)/N$ is the overlap of $\bx,\by$.
At inverse temperature $\beta \ge 0$, the SK model Gibbs measure $\mu_{N,\beta}$ is the measure supported on $\Sigma_N = \{\pm 1\}^N$ with
\balnn
  \mu_{N,\beta}(\bx) &= \fr{e^{\beta H_N(\bx)}}{2^N Z_{N,\beta}}\,, & 
  Z_{N,\beta} = \fr{1}{2^N} \sum_{\bx \in \Sigma_N} e^{\beta H_N(\bx)}\,.
\ealnn
Let $\la \cdot \ra_\beta$ denote average with respect to $\mu_{N,\beta}$, conditional on $\bW$.
Define the two-replica overlap $R_{1,2} = R(\bx^1,\bx^2)$, where $\bx^1,\bx^2$ are independent samples from $\mu_{N,\beta}$.

In this paper, we study the quenched (i.e. conditional on $\bW$) distribution of $R_{1,2}$.
At high temperature, the overlap distribution is Gaussian.
Indeed, for fixed $\beta < 1$, Guerra and Toninelli \cite{guerra2002central} showed that
\beq\label{e:GT}
  \la \delta((N(1-\beta^2))^{1/2} R_{1,2}) \ra_\beta
  \rightarrow \cN(0,1)
\eeq
in probability.
Here $\delta(x)$ denotes the Dirac delta at $x\in \bbR$, so that the left-hand side of \eqref{e:GT} is a $\bW$-measurable, probability measure-valued random variable.
We remark that \cite{guerra2002central} also addresses the more general setting where the SK model has an external field, and its main result covers a high temperature region above the Almeida--Thouless line.

Talagrand \cite[Chapter 11]{talagrand2011mean2} studied in detail the behavior of overlaps near the critical inverse temperature $\beta=1$.
He showed \cite[Theorems 11.7.1--11.7.2]{talagrand2011mean2} that at $\beta = 1 - c_N N^{-1/3}$, \eqref{e:GT} continues to hold if $c_N \rightarrow +\infty$, but not if $c_N \rightarrow c \in (0,+\infty)$.
Thus, for $\beta$ close to $1$, it is expected that the (suitably rescaled) overlap distribution has a nontrivial limit, which is a genuinely random and non-Gaussian probability measure.

Our main result characterizes this limit at criticality $\beta=1$.
Our result holds for both the SK model and the spherical SK model \cite{kosterlitz1976spherical}, whose Gibbs measure $\mu^\sph_{N,\beta}$ is supported on $S_N = \sqrt{N}\bbS^{N-1}$ with
\balnn
  \de \mu^\sph_{N,\beta}(\bx) &= \fr{e^{\beta H_N(\bx)}}{Z^\sph_{N,\beta}} \,\de \nu_N(\bx)\,, & 
  Z^\sph_{N,\beta} &= \int_{S_N} e^{\beta H_N(\bx)} \,\de \nu_N(\bx)\,.
\ealnn
Here $\nu_N$ is the normalized surface measure on $S_N$.
Let $\la \cdot \ra^\sph_\beta$ denote quenched average with respect to $\mu^\sph_{N,\beta}$.
We will abbreviate $\la \cdot \ra = \la \cdot \ra_{\beta=1}$ and $\la \cdot \ra^\sph = \la \cdot \ra^\sph_{\beta=1}$.

We next define our notion of convergence for probability measure-valued random variables.
Let $\cP_2(\bbR)$ be the space of square-integrable probability measures on $\bbR$, and $\cP_2(\cP_2(\bbR))$ be the space of probability measures $\zeta$ on $\cP_2(\bbR)$ where $\bbE_{\mu \sim \zeta} \|\mu\|_{L^2}^2 < \infty$.
We metrize $\cP_2(\bbR)$ with the $2$-Wasserstein metric $\bbW_2$, and $\cP_2(\cP_2(\bbR))$ with the $2$-Wasserstein metric $\cW_2$ relative to this metric on $\cP_2(\bbR)$.
Explicitly, for $\zeta, \zeta' \in \cP_2(\cP_2(\bbR))$,
\[
  \cW_2^2(\zeta,\zeta')
  = \inf\lt\{
    \bbE_{(\mu,\mu') \sim \pi} \bbW_2^2(\mu,\mu') : 
    \text{$\pi$ coupling of $\zeta,\zeta'$}
  \rt\}\,.
\]
The limiting overlap distribution is defined in terms of the reflected $\Airy_1$ point process.
A realization of this process is a sequence $\chi = (\chi_k)_{k\ge 1}$ of real numbers such that $\chi_1 < \chi_2 < \cdots$.
The main relevant property of this process is that for $\lambda_1 \ge \cdots \ge \lambda_N$ the eigenvalues of $\bW \sim \GOE(N)$,
\beq\label{e:airy-convergence}
  \sum_{k=1}^N \delta(N^{2/3}(2 - \lambda_k))
  \rightarrow
  \sum_{k=1}^\infty \delta(\chi_k)
\eeq
in distribution with respect to the vague topology.
See \cite[Chapter 7.8]{forrester2010log} for a textbook treatment.
Our main result is the following.
\begin{thm}\label{t:main}
  Let $\fP_{a(\chi)} \in \cP_2(\bbR)$ be defined in Definitions~\ref{d:limit-distribution} and \ref{d:a} below.
  This is a probability measure-valued random variable, measurable with respect to a realization $\chi$ of the reflected $\Airy_1$ point process, whose law is in $\cP_2(\cP_2(\bbR))$ by Proposition~\ref{p:limit-distribution-square-integrable-random}.
  As $N\to\infty$, we have
  \begin{enumerate}[label=(\alph*)]
    \item \label{i:main-cube} $\Law(\la \delta (N^{1/3} R_{1,2}) \ra)
    \stackrel{\cW_2}{\longrightarrow}
    \Law(\fP_{a(\chi)})$.
    \item \label{i:main-sphere} $\Law(\la \delta (N^{1/3} R_{1,2}) \ra^\sph)
    \stackrel{\cW_2}{\longrightarrow}
    \Law(\fP_{a(\chi)})$.
  \end{enumerate}
\end{thm}
\noindent The following corollary affirmatively resolves \cite[Conjecture 11.7.5]{talagrand2011mean2}.
\begin{cor}
  \label{c:main}
  Let $V_{a(\chi)}$ be defined in Proposition~\ref{p:limit-distribution-square-integrable-deterministic} and Definition~\ref{d:a} below.
  This is a positive real-valued random variable, measurable with respect to $\chi$, with $\bbE_\chi V_{a(\chi)} < \infty$ by Proposition~\ref{p:limit-distribution-square-integrable-random}.
  Then,
  \[
    \lim_{N\to\infty} N^{2/3} \bbE \la R_{1,2}^2 \ra
    = \lim_{N\to\infty} N^{2/3} \bbE \la R_{1,2}^2 \ra^\sph
    = \bbE_\chi V_{a(\chi)}\,.
  \]
\end{cor}
\noindent We next formally define the limiting objects $\fP_{a(\chi)}$, $V_{a(\chi)}$.
\begin{ppn}[{\cite[Theorem 6.1]{landon2022fluctuations}}]
  \label{p:Xi}
  Let $\chi = (\chi_k)_{k\ge 1}$ be a realization of the reflected $\Airy_1$ point process.
  For $k\ge 1$, define $d_k = \chi_k - \chi_1$ and $t_k = (\fr32 \pi k)^{2/3}$.
  Then the (random) limit
  \[
    \Xi(\chi) = \lim_{k\to\infty} \lt(\sum_{j=2}^k \fr{1}{d_j} - \fr{1}{\pi} \int_0^{t_k} x^{-1/2} \,\de x\rt)
  \]
  exists $\chi$-almost surely.
\end{ppn}
\begin{ppn}[Proved in \S\ref{s:limit-distribution}]
  \label{p:limit-distribution}
  Let $\chi, d_k, t_k, \Xi(\chi)$ be as in Proposition~\ref{p:Xi}.
  The following holds $\chi$-almost surely.
  \begin{enumerate}[label=(\alph*)]
    \item \label{i:limit-distribution-chi} We have $\lim_{k\to\infty} \chi_k / t_k = 1$.
    \item \label{i:limit-distribution-Psi} For $x>0$, define
    \beq\label{e:Psi}
      \Psi(x;\chi) = \fr{1}{x} + \Xi(\chi) + \sum_{k\ge 2} \lt(\fr{1}{x+d_k} - \fr{1}{d_k}\rt)\,.
    \eeq
    This function is well-defined and strictly decreasing on $(0,+\infty)$, with $\lim_{x\downarrow 0} \Psi(x;\chi) = +\infty$ and $\lim_{x\uparrow +\infty} \Psi(x;\chi) = -\infty$.
    Consequently, $\Psi(\cdot;\chi)$ has a unique positive zero $\Delta(\chi)$.
  \end{enumerate}
\end{ppn}
\begin{dfn}\label{d:limit-distribution}
  Let $t_k$ be as in Proposition~\ref{p:Xi} and $a = (a_k)_{k\ge 1}$ be a decreasing positive sequence satisfying
  \beq\label{e:a-asymptotic}
    \lim_{k\to\infty} a_kt_k = 1\,.
  \eeq
  Let $g = (g_k)_{k\ge 1}$ be a sequence of independent standard Gaussians.
  Let $\nu_a$ be the probability measure constructed in \S\ref{ss:canonical-pin} (see \eqref{e:canonical-pin-definition}), which we interpret as the distribution of $g$ conditioned on
  \beq\label{e:infinite-dimensional-condition}
    \sum_{k\ge 1} a_k (g_k^2 - 1) = 0\,.
  \eeq
  This construction is nontrivial because the constraint \eqref{e:infinite-dimensional-condition} is a null event involving infinitely many coordinates of $g$.
  For independent samples $g^{(1)}, g^{(2)} \sim \nu_a$, let
  \balnn
    Q_a &= \sum_{k\ge 1} a_k g^{(1)}_k g^{(2)}_k\,, &
    \fP_a &= \Law(Q_a)\,.
  \ealnn
\end{dfn}
\begin{ppn}[Proved in \S\ref{s:limit-distribution}]
  \label{p:limit-distribution-square-integrable-deterministic}
  For any sequence $a$ satisfying \eqref{e:a-asymptotic}, we have $\fP_a \in \cP_2(\bbR)$.
  In particular,
  \[
    V_a = \|\fP_a\|_{L^2}^2 = \sum_{k\ge 1} a_k^2 \Big[\bbE_{\nu_a} [g_k^2]\Big]^2
  \]
  is well-defined.
\end{ppn}
\begin{dfn}\label{d:a}
  Let $\chi$ be a realization of the reflected $\Airy_1$ point process.
  Define $a = a(\chi) = (a_k)_{k\ge 1}$ by 
  \[
    a_k = (\Delta(\chi) + d_k)^{-1}\,.
  \]
  By Proposition~\ref{p:limit-distribution}\ref{i:limit-distribution-chi}, $a$ satisfies \eqref{e:a-asymptotic} $\chi$-almost surely.
\end{dfn}
\begin{ppn}[Proved in \S\ref{s:overlap}]
  \label{p:limit-distribution-square-integrable-random}
  We have $\bbE_\chi V_{a(\chi)} < \infty$, and therefore $\Law(\fP_{a(\chi)}) \in \cP_2(\cP_2(\bbR))$.
\end{ppn}

\subsection{Related work}

Since the introduction of the SK model in \cite{sherrington1975solvable}, a central question has been to identify the limiting free energy density $\lim_{N\to\infty} \fr{1}{N} \bbE \log Z_{N,\beta}$.
The Parisi formula \cite{parisi1979infinite,parisi1983order} for this limit was proved (in both Ising and spherical cases, and for the more general mixed $p$-spin Hamiltonian) by Talagrand \cite{talagrand2006parisi,talagrand2006spherical} and Panchenko \cite{panchenko2013parisi}, following decades of progress in the probability and statistical physics communities \cite{parisi1979infinite,parisi1983order,mezard1987spin,ruelle1987mathematical,crisanti1992spherical,ghirlanda1998general,aizenman2003extended,guerra2002thermodynamic,guerra2003broken}.

A closely related line of work studies the distribution of fluctuations of the free energy $F_{N,\beta} = \log Z_{N,\beta}$ of the (spherical and Ising) SK model.
For the spherical SK model, these fluctuations have been characterized in the high and low temperature phases $\beta < 1$ and $\beta > 1$ \cite{baik2016fluctuations}, at criticality $\beta=1$ \cite{landon2022free}, and in a critical window around $\beta=1$ \cite{johnstone2024spin}.
For the Ising SK model at high temperature $\beta < 1$, the early work of \cite{aizenman1987some,comets1995sherrington} showed that $\Var(F_{N,\beta}) = O(1)$ and $F_{N,\beta}$ satisfies a Gaussian CLT.
At low temperature $\beta > 1$, identifying the scale of $\Var(F_{N,\beta})$ is a significant open problem, with the best upper bound $\Var(F_{N,\beta}) = O(N / \log N)$ due to Chatterjee \cite{chatterjee2009disorder}.
At $\beta = 1$, \cite{aspelmeier2008free} predicted
\beq\label{e:FE-var}
  \Var(F_{N,\beta=1}) = \fr16 \log N + O(1)\,.
\eeq
Following progressively tighter bounds on this variance from \cite{chatterjee2009disorder,talagrand2010mean,talagrand2011mean2,chen2019order,dey2026fluctuations,schertzer2026order}, the authors \cite{du2026fluctuations} proved this prediction, along with a Gaussian CLT for $F_{N,\beta=1}$; see also \cite{chen2026moderate} for a recent independent proof of \eqref{e:FE-var}.

The theory surrounding the Parisi formula also yields information about overlap distributions.
For generic mixed $p$-spin models, the minimizer of the Parisi functional is the weak limit of the annealed overlap distribution $\bbE \la \delta(R_{1,2}) \ra \in \cP([-1,1])$ \cite[Chapter 3.7]{panchenko2013sherrington}.
The genericity assumption is important for this general theorem, and the SK model is not generic in this sense.

Finer-grained questions concern the quenched overlap distribution $\la \delta(R_{1,2}) \ra$ or the scale on which the overlap fluctuates.
For the Ising SK model, the aforementioned works \cite{guerra2002central, talagrand2011mean2} show Gaussian fluctuations for $\beta = 1 - c_N N^{-1/3}$ where $c_N \rightarrow \infty$, but not if $c_N \rightarrow c \in (0,\infty)$.
For the spherical SK model with no external field at $\beta = 1 - cN^{-1/3 + \tau}$, where $c>0$, $\tau \in (0, 1/3)$ are constant, Nguyen and Sosoe \cite{nguyen2019central} showed the quenched overlap distribution $\la \delta((N(1-\beta^2))^{1/2} R_{1,2}) \ra^\sph$ converges in probability to a standard Gaussian.
At fixed $\beta > 1$, Landon and Sosoe \cite{landon2022fluctuations} showed that the quenched distribution of $R_{1,2}$ concentrates near deterministic values $\pm (1-\beta^{-1})$ and identified the $N^{-1/3}$-scale fluctuations of $\la R_{1,2}^2\ra^\sph$ and $\la |R_{1,2}| \ra^\sph$ in terms of the $\Airy_1$ point process.
Finally, recent work by the authors \cite{du2026fluctuations} showed that at criticality $\beta = 1$, $\bbE\la R_{1,2}^2\ra \asymp \bbE\la R_{1,2}^2\ra^\sph \asymp N^{-2/3}$, but does not identify the constant factor characterized in Corollary~\ref{c:main}.

\subsection{Proof ideas: overlap distribution of spherical model}
\label{ss:intro-spherical}

We first explain Theorem~\ref{t:main}\ref{i:main-sphere}, which describes the overlap distribution of the spherical SK model.
Due to the rotational invariance of the spherical SK model, this result follows from GOE spectral edge considerations, and is related to a random matrix interpretation of the spherical SK model developed in \cite{baik2016fluctuations,baik2018ferromagnetic,nguyen2019central,landon2022fluctuations,landon2022free,johnstone2024spin}.

In this model, we may assume without loss that $\bW = \diag(\lambda_1,\ldots,\lambda_N)$ is diagonal, with eigenvalues $\lambda_1 \ge \cdots \ge \lambda_N$.
Let $\gamma = \gamma(\bW) > \lambda_1$ be the unique solution to
\beq\label{e:gamma-def}
  \sum_{i=1}^N \fr{1}{\gamma - \lambda_i} = N\,.
\eeq
Then, the anisotropic Gaussian vector $\bxi \sim \cN(0, (\gamma \bI - \bW)^{-1})$ satisfies
\[
  \bbE \Big[\|\bxi\|^2\Big] = \Tr\Big[(\gamma \bI - \bW)^{-1}\Big]
  = N\,.
\]
Since $\bxi$ has density proportional to $\exp(-\fr12 ((\gamma \bI - \bW) \bx, \bx))$ over $\bx \in \bbR^N$, $\mu^\sph_{N,\beta=1}$ is precisely the law of $\bxi$ conditioned on $\|\bxi\|^2 = N$.

We next explain how the limit distribution in Theorem~\ref{t:main} arises from this description.
Write
\balnn
  \chi_{N,k} &= N^{2/3} (2 - \lambda_k)\,, &
  d_{N,k} &= \chi_{N,k} - \chi_{N,1}=N^{2/3}(\lambda_1-\lambda_k)\,, \\
  \Delta_N &= N^{2/3} (\gamma - \lambda_1)\,, &
  a_{N,k} &= (\Delta_N + d_{N,k})^{-1} = N^{-2/3} (\gamma - \lambda_k)^{-1}\,.
\ealnn
Recall from \eqref{e:airy-convergence} that the point process of the $\chi_{N,k}$ converges in distribution to a sample $\chi$ from the reflected $\Airy_1$ point process.
The equation \eqref{e:gamma-def} can be rewritten as
\[
  \fr{1}{\Delta_N} + \sum_{k=2}^N \fr{1}{\Delta_N + d_{N,k}} = N^{1/3}\,.
\]
We will see in \S\ref{ss:augmented-edge-convergence} that this converges in a suitable sense to the equation \eqref{e:Psi} defining $\Delta(\chi)$.
Consequently, $\Delta_N$ and $(a_{N,k})_{1\le k\le N}$ will also converge to $\Delta(\chi)$ and $a(\chi)$.

Let $\bg = (g_k)_{1\le k\le N}$ be a sequence of independent standard Gaussians.
We can write
\[
  \bxi = \bxi(\bg) = \sum_{k=1}^N \fr{g_k}{\sqrt{\gamma - \lambda_k}} \be_k
  = N^{1/3} \sum_{k=1}^N a_{N,k}^{1/2} g_k \be_k\,,
\]
for $\be_k$ the $k$-th basis vector in $\bbR^N$.
In light of \eqref{e:gamma-def}, $\|\bxi\|^2 = N$ is equivalent to
\beq\label{e:bxi-condition}
  \sum_{k=1}^N a_{N,k} (g_k^2-1) = 0\,.
\eeq
For $j=1,2$, let $\bg^{(j)} = (g^{(j)}_k)_{1\le k\le N}$ be independent samples of the conditional law of $\bg$ given \eqref{e:bxi-condition}.
Then for $\bx^1,\bx^2$ independent samples from $\mu^\sph_{N,\beta=1}$ conditional on $\bW$, we have
\[
  N^{1/3} R(\bx^1,\bx^2)
  \stackrel{d}{=} N^{-2/3} \Big(\bxi(\bg^{(1)}), \bxi(\bg^{(2)})\Big)
  = \sum_{k=1}^N 
  a_{N,k} g^{(1)}_k g^{(2)}_k\,.
\]
This is the finite-$N$ version of the limiting random variable $Q_a$ from Definition~\ref{d:limit-distribution}.
To prove Theorem~\ref{t:main}\ref{i:main-sphere}, we will show that the above objects are continuous in the appropriate topologies so that they pass to their limits as $N\to\infty$.

\subsection{Proof ideas: sphere-to-cube comparison}

While the spherical SK model's overlap distribution has a natural random matrix description due to its rotational invariance, no such description is a priori available for the Ising SK model.
To identify the latter model's overlap distribution, we show that the two critically rescaled overlap distributions coincide in the limit.
\begin{thm}\label{t:sphere-to-cube}
  We have
  \[
    \lim_{N\to\infty} \cW_2\Big(
      \Law(\la \delta (N^{1/3} R_{1,2})\ra),
      \Law(\la \delta (N^{1/3} R_{1,2})\ra^\sph)
    \Big) = 0\,.
  \]
\end{thm}

Theorem~\ref{t:sphere-to-cube} builds on a similar transfer principle for the partition function introduced in recent work by the authors.
Let $Z_N = Z_{N,\beta=1}$, $Z^\sph_N = Z^\sph_{N,\beta=1}$, and $X_N = Z_N / Z^\sph_N$.
Then, \cite[Theorem 1.6]{du2026fluctuations} shows
\beq\label{e:overview-sphere-to-cube}
  \bbE[(X_N - 1)^2] = O(N^{-1/3})\,.
\eeq
This is used in \cite{du2026fluctuations} to infer the critical Ising SK model's free energy CLT from the analogous CLT for the spherical SK model \cite{landon2022free, johnstone2024spin}.\footnote{Ideas in the proof of \eqref{e:overview-sphere-to-cube} are also used in \cite{du2026fluctuations} to prove the variance asymptotic \eqref{e:FE-var}. The latter proof is not a transfer from the spherical model, and involves showing new asymptotically sharp bounds on the SK model's annealed overlap moments.}
We next outline the proof of \eqref{e:overview-sphere-to-cube} and explain how these ideas extend to prove Theorem~\ref{t:sphere-to-cube}.
For $\bx,\by \in S_N$ with $R(\bx,\by) = q \in [-1,1]$, define
\[
  J_N(q) = \bbE\lt[\fr{e^{H_N(\bx) + H_N(\by)}}{(Z_N^\sph)^2}\rt]\,.
\]
By rotational invariance, this depends on $\bx,\by$ only through $q$, so the notation $J_N(q)$ is justified.
Let $\bbE_q$ and $\bbE^\sph_q$ denote expectation with respect to $q = R(\bx,\by)$ for, respectively, $\bx,\by \sim \unif(\Sigma_N)$ and $\bx,\by \sim \unif(S_N)$.
A direct calculation \cite[Lemma 2.1]{du2026fluctuations} shows that for any test function $g : [-1,1] \rightarrow \bbR$ for which the moments below are defined,
\baln\label{e:reweighted-mt}
  \bbE \lt[X_N^2 \la g(R_{1,2}) \ra\rt] &= \bbE_q[J_N(q)g(q)]\,, & 
  \bbE \lt[\la g(R_{1,2}) \ra^\sph\rt] &= \bbE^\sph_q[J_N(q)g(q)]\,.
\ealn
Since $Z^\sph_N$ is the average of $Z_N$ over orthogonal rotations of $\bW$, we also have $\bbE[X_N] = 1$.
Hence,
\[
  \bbE[(X_N-1)^2] = \bbE[X_N^2] - 1 = \bbE_q[J_N(q)] - \bbE^\sph_q[J_N(q)]\,.
\]
Thus \eqref{e:overview-sphere-to-cube} amounts to comparing the expectations of $J_N(q)$ under $\bbE_q$ and $\bbE^\sph_q$.
This is achieved using two facts.
First, the main contribution to both expectations comes from the scale $|q| \lesssim N^{-1/3}$.
Second, the distributions of $q$ under $\bbE_q$ and $\bbE^\sph_q$ nearly agree for $q$ in this range, in the following sense.
Under $\bbE^\sph_q$, $q$ has density
\[
  \rho_N(q) \,\,\propto\,\, (1-q^2)^{(N-3)/2} = \exp\lt(-\fr{N}{2} q^2 + O(Nq^4 + q^2)\rt)\,, 
\]
and under $\bbE_q$, $q$ is sampled from the discrete distribution on $\{-1,-1+\fr2N,\ldots,1\}$ with probability mass
\[
  p_N(q) = 2^{-N} \binom{N}{\fr{1+q}{2}N} \,\,\propto\,\, \exp\lt(-\fr{N}{2} q^2 + O(Nq^4 + q^2)\rt)\,.
\]
The error terms $O(Nq^4 + q^2)$ are $O(N^{-1/3})$ for $|q| \lesssim N^{-1/3}$.
After further approximating the discrete sum over $q$ in $\bbE_q$ with a continuous integral, \cite{du2026fluctuations} obtains \eqref{e:overview-sphere-to-cube}.

The proof of Theorem~\ref{t:sphere-to-cube} extends the above ideas, and is based on showing the quenched distributions' Laplace transforms coincide.
We will show that for a universal $t_0 > 0$, and all $|t| < t_0$,
\beq\label{e:laplace-coincide}
  \bbE \lt|
    X_N^2 \la \exp(tN^{1/3} R_{1,2}) \ra 
    - \la \exp(tN^{1/3} R_{1,2}) \ra^\sph
  \rt| = o_N(1)\,.
\eeq
Recall that \eqref{e:overview-sphere-to-cube} implies that $X_N \rightarrow 1$ in probability.
Then, \eqref{e:laplace-coincide} implies that the distance between $\la \delta(N^{1/3} R_{1,2}) \ra$ and $\la \delta(N^{1/3} R_{1,2}) \ra^\sph$ in any metric metrizing the weak topology tends to zero in probability.
Furthermore, \cite[Theorem 1.4(a) and Corollary 1.7(b)]{du2026fluctuations} show that $N^{1/3} R_{1,2}$ has a bounded exponential moment under both $\bbE \la \cdot \ra$ and $\bbE \la \cdot \ra^\sph$.
This suffices to upgrade the above convergence to $\cW_2$.

We finally explain the proof of \eqref{e:laplace-coincide}.
A comparison argument similar to above shows the weaker bound
\beq\label{e:laplace-coincide-weak}
  \lt|
    \bbE \lt[X_N^2 \la \exp(tN^{1/3} R_{1,2}) \ra\rt]
    - \bbE \lt[\la \exp(tN^{1/3} R_{1,2}) \ra^\sph\rt]
  \rt| = o_N(1)
\eeq
for all $|t| < t_0$.
Indeed, by \eqref{e:reweighted-mt},
\[
  \bbE \lt[X_N^2 \la \exp(tN^{1/3} R_{1,2}) \ra\rt]
  - \bbE \lt[\la \exp(tN^{1/3} R_{1,2}) \ra^\sph\rt]
  = \bbE_q[J_N(q) e^{tN^{1/3}q}]
  - \bbE^\sph_q[J_N(q) e^{tN^{1/3}q}]\,.
\]
For $|t| < t_0$, the dominant contributions to these expectations still come from $|q|\lesssim N^{-1/3}$, and the comparison method described above bounds this difference.

Finally, we upgrade \eqref{e:laplace-coincide-weak} to \eqref{e:laplace-coincide} using an idea due to GPT-5.6 Pro: for any $t\in [0, t_0)$, we may define a Hilbert space $\cH_N$ and kernel $\Phi_{N,t}: \bbR^N \rightarrow \cH_N$ with kernel inner product
\[
  (\Phi_{N,t}(\bx), \Phi_{N,t}(\by))_{\cH_N}
  = \exp\lt(tN^{1/3} R(\bx,\by)\rt)\,.
\]
In particular,
\balnn
  \|\la \Phi_{N,t}(\bx) \ra\|^2_{\cH_N} &= \la \exp(tN^{1/3} R_{1,2}) \ra\,, & 
  \|\la \Phi_{N,t}(\bx) \ra^\sph\|^2_{\cH_N} &= \la \exp(tN^{1/3} R_{1,2}) \ra^\sph\,.
\ealnn
Then a direct calculation shows that
\[
  \bbE \lt[\lt\| X_N \la \Phi_{N,t}(\bx) \ra - \la \Phi_{N,t}(\bx) \ra^\sph \rt\|^2_{\cH_N}\rt]
  = \bbE \lt[X_N^2 \la \exp(tN^{1/3} R_{1,2}) \ra\rt]
  - \bbE \lt[\la \exp(tN^{1/3} R_{1,2}) \ra^\sph\rt]\,,
\]
which is $o_N(1)$ by \eqref{e:laplace-coincide-weak}.
Using the estimate, for any square-integrable random vectors $\bu,\bv$,
\[
  \bbE \lt|\|\bu\|^2 - \|\bv\|^2\rt| 
  \le \bbE \lt[\|\bu - \bv\|^2\rt]^{1/2} 
  \bbE \lt[\|\bu + \bv\|^2\rt]^{1/2} 
  \le \bbE \lt[\|\bu - \bv\|^2\rt]^{1/2} 
  \cdot 2^{1/2} \bbE \lt[\|\bu\|^2 + \|\bv\|^2\rt]^{1/2}\,,
\]
we deduce
\balnn
  &\bbE \lt[\lt|
    X_N^2 \la \exp(tN^{1/3} R_{1,2}) \ra 
    - \la \exp(tN^{1/3} R_{1,2}) \ra^\sph
  \rt|\rt]
  = \bbE \lt|X_N^2 \|\la \Phi_{N,t}(\bx) \ra\|^2_{\cH_N} - \|\la \Phi_{N,t}(\bx) \ra^\sph\|^2_{\cH_N}\rt| \\
  &\le \bbE \lt[\lt\| X_N \la \Phi_{N,t}(\bx) \ra - \la \Phi_{N,t}(\bx) \ra^\sph \rt\|^2_{\cH_N}\rt]^{1/2}
  \cdot 2^{1/2} \bbE \lt[X_N^2 \|\la \Phi_{N,t}(\bx) \ra\|^2_{\cH_N} + \|\la \Phi_{N,t}(\bx) \ra^\sph\|^2_{\cH_N}\rt]^{1/2} \\
  &= o_N(1)
\ealnn
for all $t\in [0, t_0)$.
Reflection symmetry of the quenched overlap distributions across $0$ proves \eqref{e:laplace-coincide}.

\subsection{On the critical window}

Consider the critical window $\beta_N = 1 + bN^{-1/3}$, where $b\in \bbR$ is independent of $N$.
A heuristic argument similar to that of \S\ref{ss:intro-spherical} suggests that
\beq\label{e:critical-window}
\Law(\la \delta(N^{1/3} R_{1,2}) \ra^\sph_{\beta_N})
\stackrel{\cW_2}{\longrightarrow}
\Law(\fP_{a^b(\chi)})\,,
\eeq
where $a^b(\chi)$ is defined as follows.
Let $\Psi(\cdot;\chi)$ be defined in \eqref{e:Psi} and $\Delta^b(\chi) > 0$ be the (unique) solution to $\Psi(\Delta^b(\chi);\chi) = b$.
Then define $a^b(\chi) = (a^b_k)_{k\ge 1}$ by $a^b_k = (\Delta^b(\chi) + d_k)^{-1}$.

As $\beta$ ranges over this critical window, $\fP_{a^b(\chi)}$ transitions from Gaussian as $b\downarrow-\infty$ to bimodal as $b\uparrow+\infty$.
Indeed, it is not hard to check that for almost all $\chi$,
\balnn
\Law\lt((2|b|)^{1/2} x: x \sim \fP_{a^b(\chi)}\rt) &\stackrel{b\downarrow-\infty}{\longrightarrow} \cN(0,1)\,, &
\Law\lt(b^{-1} x: x \sim \fP_{a^b(\chi)}\rt) &\stackrel{b\uparrow+\infty}{\longrightarrow} \fr{\delta(-1)+\delta(1)}{2}\,.
\ealnn
Our proof of Theorem~\ref{t:main}\ref{i:main-sphere} can be adapted to prove that the spherical model satisfies \eqref{e:critical-window} throughout this critical window.
For reasons outlined in \cite[Remark 2.3]{du2026fluctuations}, we also expect the sphere-to-cube transfer principle from Theorem~\ref{t:sphere-to-cube} to remain true in this range of $\beta_N$, so that the analogous limit with \(\la \cdot \ra_{\beta_N}\) holds for the Ising SK model's overlap distribution as well.
However, as described therein, our proof of Theorem~\ref{t:sphere-to-cube} will not work for $b > 0$.
For simplicity, we focus on $\beta=1$ in this paper.

Finally, note that this critical window is different from the critical window $\beta_N = 1 + b N^{-1/3} \sqrt{\log N}$ for the free energy fluctuations, on which \cite{johnstone2024spin} showed that the spherical SK model's free energy fluctuations transition from Gaussian to Tracy--Widom.
We expect the Ising SK model's free energy fluctuations to have the same distributional limit in this critical window, but this has not been proved; see \cite[Remark 2.3]{du2026fluctuations}.

\subsection{Organization}

The rest of this paper is structured as follows.
\begin{itemize}
    \item In \S\ref{s:limit-distribution} we prove Propositions~\ref{p:limit-distribution} and \ref{p:limit-distribution-square-integrable-deterministic} and formally construct the conditioning of $g$ in Definition~\ref{d:limit-distribution}.
    This constructs the limiting objects $\fP_{a(\chi)}$, $V_{a(\chi)}$ in Theorem~\ref{t:main} and Corollary~\ref{c:main} (though it does not show $\Law(\fP_{a(\chi)}) \in \cP_2(\cP_2(\bbR))$, which is a consequence of Proposition~\ref{p:limit-distribution-square-integrable-random}).
    \item In \S\ref{s:sphere-to-cube} we prove Theorem~\ref{t:sphere-to-cube}, which shows that the critically rescaled spherical and Ising overlap distributions coincide in the limit.
    \item In \S\ref{s:overlap} we formalize the limiting argument outlined in \S\ref{ss:intro-spherical}. This leads to proofs of Proposition~\ref{p:limit-distribution-square-integrable-random} and Theorem~\ref{t:main}\ref{i:main-sphere}.
    We finally infer Theorem~\ref{t:main}\ref{i:main-cube} using the sphere-to-cube comparison from Theorem~\ref{t:sphere-to-cube}.
\end{itemize}

\subsection*{Acknowledgments}

We are extremely grateful to Jason Prodromidis for several insightful and motivating conversations during the early stages of this project.
We also thank Wei-Kuo Chen for bringing this problem to our attention and Fu-Hsuan Ho for an inspiring discussion.
HD was partially supported by an NSF-Simons research collaboration grant (award number 2031883).
BH was supported by a Stanford Science Fellowship and an NSF Mathematical Sciences Postdoctoral Fellowship.

\subsection*{The role of AI in this work}

This paper is a companion to \cite{du2026fluctuations}, and is intended to explore further consequences of ideas developed therein with AI assistance.
Most of the formal arguments in this paper were initially generated by GPT-5.6 Pro.

The principal human inputs were: the identification of the limiting objects $\fP_{a(\chi)}$ and $V_{a(\chi)}$ through the heuristic argument in \S\ref{ss:intro-spherical}; the formulation of Theorem~\ref{t:sphere-to-cube} providing the sphere-to-cube comparison; and the overall proof strategy of first establishing the limiting overlap distribution for the spherical model and then transferring the result to the Ising model.
The authors take full responsibility for the correctness of the paper.

\section{Formal construction of the limit distribution}
\label{s:limit-distribution}

In this section we construct the limiting objects $\fP_{a(\chi)}$ and $V_{a(\chi)}$ in Theorem~\ref{t:main} and Corollary~\ref{c:main}.
The section is organized as follows.
\begin{itemize}
  \item In \S\ref{ss:limit-airy-saddle}, we prove Proposition~\ref{p:limit-distribution}, which provides almost-sure asymptotics of the Airy points $\chi_k$ and shows $\Delta(\chi)$ is well-defined as the zero of $\Psi(\cdot;\chi)$ from \eqref{e:Psi}.
  \item In \S\ref{ss:canonical-pin}, we formally construct the conditioning $\nu_a$ of the infinite-dimensional Gaussian sequence $g$ in Definition~\ref{d:limit-distribution}.
  \item In \S\ref{ss:limit-square-integrability}, we prove Proposition~\ref{p:limit-distribution-square-integrable-deterministic}, which gives square integrability of $\fP_a$.
  \item In \S\ref{ss:airy-regularity}, we establish regularity properties of the conditioned measure $\nu_a$ defined in \S\ref{ss:canonical-pin}.
  These will be used in \S\ref{s:overlap} to compare $\nu_a$ with the finite-dimensional conditioned measures introduced in \S\ref{ss:intro-spherical}.
\end{itemize}

\subsection{The limiting Airy saddle}
\label{ss:limit-airy-saddle}

In this subsection we prove Proposition~\ref{p:limit-distribution}.

\begin{proof}[Proof of Proposition~\ref{p:limit-distribution}\ref{i:limit-distribution-chi}]
  The estimates \cite[Equations (6.47)--(6.48)]{landon2022fluctuations} imply that for some universal constants $C_1, C > 0$, all integers $k\ge 1$, and all real 
  \beq\label{e:LS22-s-constraint}
    s \in (Ck^{-1/3}, k^{2/3} - C_1)\,,
  \eeq
  we have
  \beq\label{e:airy-rigidity-input}
    \bbP\lt(|\chi_k-t_k|>s\rt)
    \le
    \fr{C(\log k+\log(1+s))}{(k^{1/3}s-C)^2}\,.
  \eeq
  Let $(\eps_k)_{k\ge 1}$ be a sequence tending to $0$ sufficiently slowly with $k$.
  Since $t_k \asymp k^{2/3}$, there exists $k_0$ such that $s = \eps_k t_k$ satisfies \eqref{e:LS22-s-constraint} for all $k\ge k_0$.
  Then, by \eqref{e:airy-rigidity-input},
  \[
    \sum_{k\ge k_0} \bbP\lt(|\chi_k-t_k|>\eps_k t_k\rt)
    \le C\sum_{k\ge k_0} \fr{\log k+\log(1+\eps_k t_k)}{(k^{1/3}\eps_k t_k-C)^2}\,.
  \]
  For $\eps_k$ tending to $0$ sufficiently slowly, this upper bound is finite.
  The Borel--Cantelli lemma implies that $|\chi_k-t_k|\le \eps_k t_k$ for all sufficiently large $k$ almost surely.
  The conclusion follows.
\end{proof}

\begin{proof}[Proof of Proposition~\ref{p:limit-distribution}\ref{i:limit-distribution-Psi}]
  Fix a realization of $\chi$ on which the conclusions of Propositions~\ref{p:Xi} and \ref{p:limit-distribution}\ref{i:limit-distribution-chi} hold.
  Since $\chi_1$ is finite, the latter conclusion implies 
  \beq\label{e:dk-asymptotic}
    \lim_{k\to+\infty} d_k/t_k = 1\,.
  \eeq
  This implies that as $x\to+\infty$,
  \[
    N_d(x)\equiv \#\{k\ge 2:d_k\le x\}
    = \fr{2}{3\pi}x^{3/2}(1+o_x(1))\,,
  \]
  where $o_x(1)$ denotes a ($\chi$-dependent) term tending to $0$ in this limit.
  Thus there exist $x_0, c$ (depending on $\chi$) such that for all $x \ge x_0$,
  \[
    N_d(x) \ge cx^{3/2}\,.
  \]
  The asymptotic \eqref{e:dk-asymptotic} also implies $\sum_{k\ge 2} d_k^{-2} < \infty$.
  On any compact interval
  $[L_-,L_+]\subset(0,+\infty)$,
  \balnn
    \lt|\fr1{x+d_k}-\fr1{d_k}\rt|
    &=\fr{x}{d_k(x+d_k)}
    \le\fr{L_+}{d_k^2}\,, &
    \fr1{(x+d_k)^2}&\le\fr1{d_k^2}\,.
  \ealnn
  Thus the series defining $\Psi$ in \eqref{e:Psi}, and the series obtained by differentiating it, converge uniformly on $[L_-,L_+]$.
  It follows that $\Psi$ is continuously differentiable on $(0,+\infty)$ and
  \beq\label{e:Psi-derivative}
    \Psi'(x;\chi)
    = -\fr1{x^2}-\sum_{k\ge2}\fr1{(x+d_k)^2}
    < 0.
  \eeq
  This shows that $\Psi$ is strictly decreasing.
  Moreover,
  \[
    \lt|
      \sum_{k\ge2}\lt(
        \fr{1}{x+d_k} - \fr{1}{d_k}
      \rt)
    \rt|
    \le x\sum_{k\ge2} d_k^{-2}\,,
  \]
  where we recall that this sum is finite.
  Thus there exists $C$ depending on $\chi$ such that
  \[
    \Psi(x;\chi) \ge \fr1x + \Xi(\chi) - Cx\,.
  \]
  It follows that $\lim_{x\downarrow 0} \Psi(x;\chi) = +\infty$.
  It remains to consider $x\to+\infty$.
  Note that for $x\ge x_0$ defined above,
  \[
    |\Psi'(x;\chi)|
    \ge \sum_{k\ge2}\fr1{(x+d_k)^2}
    \ge \sum_{d_k\le x}\fr1{(x+d_k)^2}
    \ge \fr{N_d(x)}{4x^2}
    \ge \fr{c}{4\sqrt{x}}\,.
  \]
  Integrating \eqref{e:Psi-derivative} on $[x_0,x]$ gives
  \[
    \Psi(x;\chi)
    \le \Psi(x_0;\chi) - \fr{c}{4}\int_{x_0}^x s^{-1/2}\,\de s\,.
  \]
  This tends to $-\infty$ as $x\uparrow+\infty$, as the integral diverges.
\end{proof}

\subsection{Canonical conditioning at the constraint}
\label{ss:canonical-pin}

The goal of this subsection is to give the formal construction of $\nu_a$ required in Definition~\ref{d:limit-distribution}.
In Lemma~\ref{l:canonical-disintegration}, we define projectively consistent finite-dimensional conditional laws and use Kolmogorov extension to obtain an infinite-dimensional probability kernel $(\nu_{a,s} : s\in \bbR)$.
We then set $\nu_a = \nu_{a,0}$.

We also prove Lemma~\ref{l:constraint-fiber}, which shows that under $\nu_a$, $\sum_{k=1}^n a_k (g_k^2-1) \rightarrow 0$ almost surely and in $L^2$.
This lemma is not formally required in subsequent arguments, which use only the finite-dimensional marginals of $\nu_a$.
We include it to show that we can genuinely interpret $\nu_a$ as a conditioning of an infinite Gaussian sequence, as described in Definition~\ref{d:limit-distribution}.

In this subsection, we fix a decreasing positive sequence $a$ satisfying \eqref{e:a-asymptotic}, as in Definition~\ref{d:limit-distribution}.
This implies
\baln\label{e:a-basic-summability}
  \sum_{k\ge1}a_k^2&<\infty\,, & 
  \sum_{k\ge1}a_k&=\infty.
\ealn
Let $g_1,g_2,\ldots$ be a sequence of independent standard Gaussians, and define
\beq\label{e:Ta}
  T_a = \sum_{k\ge1}a_k(g_k^2-1).
\eeq
The series converges in $L^2$ and almost surely, because its summands are independent and centered and their variances sum to $2\sum_k a_k^2<\infty$.
Write
\balnn
  S_{a,n} &= \sum_{k\le n}a_k(g_k^2-1)\,, &
  T_{a,>n} &= \sum_{k>n}a_k(g_k^2-1)\,.
\ealnn
\begin{lem}\label{l:constraint-densities}
  For every $n\ge0$, the variables $T_a$ and $T_{a,>n}$ have probability densities $p_a$ and $p_{a,>n}$ on $\bbR$ that are bounded, continuous, and strictly positive.
\end{lem}
\noindent The proof of Lemma~\ref{l:constraint-densities} appears after the following lemma.
\begin{lem}\label{l:constraint-densities-positive}
  For any $n\ge 3$ and $c_1,\ldots,c_n>0$, 
  \[
    Y = \sum_{k=1}^n c_k g_k^2
  \]
  has a bounded continuous density which is positive on $(0,\infty)$.
\end{lem}
\begin{proof}
  Write $Y_k = c_k g_k^2$.
  Then $Y_k$ has density
  \[
    q_k(y) = \fr{\bbone\{y>0\} e^{-y/(2c_k)}}{\sqrt{2\pi c_k y}}\,,
  \]
  and the density $q$ of $Y$ is the convolution $q_1*\cdots*q_n$.
  We write this convolution more explicitly.
  Define the simplex
  \[
    \Delta_{n-1}
    = \lt\{
      (s_1,\ldots,s_{n-1})\in [0,\infty)^{n-1}:
      \sum_{k=1}^{n-1}s_k\le 1
    \rt\}\,,
  \]
  and let $s_n = 1 - \sum_{k=1}^{n-1} s_k$.
  For $y>0$, writing $y_k = ys_k$ in the convolution integral yields
  \beq\label{e:compound-chi-square-density}
    q(y)
    = \fr{y^{n/2-1}}{(2\pi)^{n/2}\prod_{k=1}^n\sqrt{c_k}}
    \int_{\Delta_{n-1}}
    \lt(\prod_{k=1}^ns_k^{-1/2}\rt)
    \exp\lt(
      -\fr y2\sum_{k=1}^n\fr{s_k}{c_k}
    \rt)
    \,\de s_1\cdots\de s_{n-1}\,.
  \eeq
  Note that the first factor in the integral is integrable, as the Dirichlet integral satisfies
  \beq\label{e:dirichlet}
    \int_{\Delta_{n-1}}
    \prod_{k=1}^ns_k^{-1/2}
    \,\de s_1\cdots\de s_{n-1}
    = \fr{\Gamma(1/2)^n}{\Gamma(n/2)}
    <\infty\,.
  \eeq
  Thus $q(y)$ is finite for all $y>0$.
  It is also clearly positive for all $y>0$ because the integrand in \eqref{e:compound-chi-square-density} is positive.
  By \eqref{e:dirichlet} and dominated convergence, for any sequence $y_m \rightarrow y \in (0,\infty)$ we have $q(y_m) \rightarrow q(y)$, and thus $q$ is continuous on $(0,\infty)$.
  Finally, \eqref{e:dirichlet} implies the existence of $C>0$ depending only on $c_1,\ldots,c_n$ such that
  \[
    q(y) \le C y^{n/2-1} e^{-y/(2c_{\max})}
  \]
  for all $y>0$, where $c_{\max} = \max(c_1,\ldots,c_n)$.
  This implies $q(y) \rightarrow 0$ as $y\downarrow 0$ or $y\uparrow +\infty$.
  Together with continuity on $(0,\infty)$, this shows $q$ is bounded.
\end{proof}
\begin{proof}[Proof of Lemma~\ref{l:constraint-densities}]
  We prove the result for $T_a$.
  The proof for $T_{a,>n}$ follows by relabeling $a_{n+i}$ to $a_i$, which preserves the assumption \eqref{e:a-asymptotic}.
  Let
  \beq\label{e:Phi}
    \Phi_a(u) = \bbE \lt[e^{iuT_a}\rt]
  \eeq
  be the characteristic function of $T_a$.
  We will first justify that
  \beq\label{e:constraint-characteristic}
    \Phi_a(u) = \prod_{k\ge1} e^{-iua_k}(1-2iua_k)^{-1/2}\,,
  \eeq
  where we use the principal square root.
  Define for use below
  \[
    L_u(v)=-iuv-\fr12\operatorname{Log}_{\mathrm{pr}}(1-2iuv)\,.
  \]
  A direct calculation shows that $S_{a,n}$ has characteristic function
  \[
    \Phi_{a,n}(u) = \bbE\lt[e^{iuS_{a,n}}\rt]
    = \prod_{k=1}^n \bbE\lt[e^{iua_k (g_k^2-1)}\rt]
    = \prod_{k=1}^n e^{-iua_k}(1-2iua_k)^{-1/2}
    = \exp\lt(\sum_{k=1}^n L_u(a_k)\rt)\,.
  \]
  As noted above, $S_{a,n} \rightarrow T_a$ in $L^2$.
  Thus, for any $u \in \bbR$,
  \[
    |\Phi_{a,n}(u) - \Phi_a(u)|
    = \lt|\bbE \lt[e^{iuS_{a,n}}\rt] - \bbE \lt[e^{iuT_a}\rt]\rt|
    \le |u| \|S_{a,n} - T_a\|_{L^2}\,.
  \]
  This implies that 
  \beq\label{e:Phin-to-Phi}
    \Phi_{a,n}(u) \rightarrow \Phi_a(u)
  \eeq
  uniformly over $u$ in any compact set.
  As $L_u(0) = 0$ and $L'_u(v) = -\fr{2u^2v}{1-2iuv}$, we also have
  \[
    |L_u(v)| \le \int_0^{|v|} \fr{2u^2t}{|1-2iut|}\,\de t
    \le u^2v^2\,.
  \]
  Together with the first estimate from \eqref{e:a-basic-summability}, we conclude that $\sum_k L_u(a_k)$ converges absolutely, uniformly over $u$ in any compact set.
  Thus, uniformly over $u$ in any compact set,
  \[
    \Phi_{a,n}(u) \rightarrow \exp\lt(\sum_{k\ge 1} L_u(a_k)\rt)
    = \prod_{k\ge1} e^{-iua_k}(1-2iua_k)^{-1/2}\,.
  \]
  Together with \eqref{e:Phin-to-Phi} this proves \eqref{e:constraint-characteristic}.
  We will next prove that there exist $C_a, c_a > 0$ such that
  \beq\label{e:Phi-decay}
    |\Phi_a(u)|\le C_a\exp(-c_a|u|^{3/2}).
  \eeq
  Taking absolute values in \eqref{e:constraint-characteristic} gives
  \beq\label{e:Phi-modulus}
    |\Phi_a(u)|
    = \prod_{k\ge1} (1+4u^2a_k^2)^{-1/4}.
  \eeq
  For $|u|$ sufficiently large, \eqref{e:a-asymptotic} implies that there are of order $|u|^{3/2}$ terms $a_k$ such that $1\le |u|a_k \le 2$.
  Restricting the product in \eqref{e:Phi-modulus} to these terms shows \eqref{e:Phi-decay}.
  Thus $\Phi_a$ is integrable.
  The Fourier inversion formula
  \[
    2\pi p_a(x) = \int_{\bbR} e^{-iux} \Phi_a(u) \,\de u
  \]
  then implies $p_a$ is bounded and continuous.
  We turn to positivity of $p_a$.
  Consider any $x\in \bbR$ and $\eps > 0$.
  Choose $n \ge 3$ large enough that
  \balnn
    A_n \equiv \sum_{k=1}^n a_k &> \eps-x\,, &
    2\sum_{k>n}a_k^2 &< \eps^2\,.
  \ealnn
  Such $n$ exists by \eqref{e:a-basic-summability}.
  Write $T_a = S_{a,n} + T_{a,>n}$.
  By Lemma~\ref{l:constraint-densities-positive}, $S_{a,n}$ has a continuous density $f_{a,n}$ which is positive on $(-A_n, +\infty)$.
  By Chebyshev's inequality, $\bbP(|T_{a,>n}|<\eps) > 0$.
  It follows that
  \[
    p_a(x) = \bbE f_{a,n}(x-T_{a,>n}) > 0\,. \qedhere
  \]
\end{proof}

\begin{lem}\label{l:canonical-disintegration}
  Let $\varphi(x) = (2\pi)^{-1/2} \exp(-x^2/2)$ be the standard Gaussian density.
  For $s\in\bbR$ and $n\ge1$, let
  \beq\label{e:canonical-pin-marginal-s}
    \nu_{a,s}^{(n)}(\de z_1\cdots\de z_n)
    = \fr{1}{p_a(s)}
    \lt(\prod_{k=1}^n\varphi(z_k)\,\de z_k\rt)
    p_{a,>n}\lt(
      s-\sum_{k=1}^na_k(z_k^2-1)
    \rt)\,.
  \eeq
  For every $s\in \bbR$, the family $(\nu_{a,s}^{(n)})_{n\ge1}$ consists of projectively consistent probability measures and therefore determines a unique probability measure $\nu_{a,s}$ on $\bbR^{\bbN}$.
  Moreover, $s\mapsto\nu_{a,s}$ is a probability kernel.
\end{lem}
\begin{proof}
  The denominator $p_a(s)$ in \eqref{e:canonical-pin-marginal-s} is positive by Lemma~\ref{l:constraint-densities}.
  Integrating the numerator gives the convolution density of $S_{a,n}+T_{a,>n}=T_a$ at $s$.
  This equals $p_a(s)$ for every $s$, because the two continuous densities agree almost everywhere.
  Thus \eqref{e:canonical-pin-marginal-s} defines a probability measure.

  For fixed $z_1,\ldots,z_{n-1}$, convolution in the last coordinate gives
  \[
    \int_{\bbR}
    \varphi(z_n)
    p_{a,>n}\lt(
      s-\sum_{k=1}^na_k(z_k^2-1)
    \rt)\,\de z_n
    = p_{a,>n-1}\lt(
      s-\sum_{k=1}^{n-1}a_k(z_k^2-1)
    \rt).
  \]
  The identity holds pointwise by continuity.
  Hence the family $(\nu_{a,s}^{(n)})_{n\ge1}$ is projectively consistent.
  Kolmogorov's extension theorem gives a probability measure $\nu_{a,s}$ on $\bbR^{\bbN}$ for every $s\in\bbR$.

  For every cylinder set $C\subset\bbR^{\bbN}$, \eqref{e:canonical-pin-marginal-s} shows that $s\mapsto\nu_{a,s}(C)$ is Borel.  The collection
  \[
    \sD_a
    = \{
      B \in \cB(\bbR^{\bbN}):
      s\mapsto\nu_{a,s}(B)\text{ is Borel}
    \}
  \]
  is a monotone class containing the cylinder sets.
  It is therefore the whole Borel $\sigma$-field of $\bbR^{\bbN}$, so $s\mapsto\nu_{a,s}$ is a probability kernel.
\end{proof}
We formally define the conditioning in Definition~\ref{d:limit-distribution} by
\beq\label{e:canonical-pin-definition}
  \nu_a=\nu_{a,0}\,.
\eeq
\begin{lem}\label{l:constraint-fiber}
  Under $\nu_a$,
  \[
    \sum_{k=1}^na_k(g_k^2-1)\rightarrow 0
  \]
  almost surely and in $L^2(\nu_a)$.
\end{lem}
\begin{proof}
  Let $f_{a,n}$ denote the density of $S_{a,n}$ under the product Gaussian law.
  The formula \eqref{e:canonical-pin-marginal-s} shows that under $\nu_a$, $S_{a,n}$ has density
  \beq\label{e:S-an-density}
    \tilde f_{a,n}(x) = \fr{f_{a,n}(x)p_{a,>n}(-x)}{p_a(0)}\,.
  \eeq
  For $n\ge 3$, write $f_{a,n}=f_{a,3}*\rho_{4:n}$, where $\rho_{4:n}$ is a probability measure.
  Lemma~\ref{l:constraint-densities-positive} implies that $\|f_{a,3}\|_\infty<\infty$, and therefore $\|f_{a,n}\|_\infty \le \|f_{a,3}\|_\infty$.
  It follows that as $n\rightarrow\infty$,
  \[
    \bbE_{\nu_a} \lt[S_{a,n}^2\rt]
    = \int x^2 \tilde f_{a,n}(x)\,\de x
    \le \fr{\|f_{a,3}\|_\infty}{p_a(0)}
         \int_{\bbR}x^2p_{a,>n}(-x)\,\de x
    =\fr{2\|f_{a,3}\|_\infty}{p_a(0)}
         \sum_{k>n}a_k^2
    \rightarrow 0\,.
  \]
  This shows the desired convergence in $L^2(\nu_a)$.
  We turn to the proof of almost-sure convergence.
  Fix $p>6$, and for $m>n$ set
  \[
    R_{n,m}=\sum_{n<k\le m}a_k(g_k^2-1)\,.
  \]
  Rosenthal's inequality gives that for $C_p$ depending only on $p$,
  \[
    \bbE|R_{n,m}|^p
    \le C_p\lt[
      \lt(\sum_{n<k\le m}a_k^2\rt)^{p/2}
      +\sum_{n<k\le m}a_k^p
    \rt],
  \]
  where the moments of $g_k^2-1$ have been absorbed into $C_p$.
  Since
  \balnn
    \sum_{k>m} a_k^2 &\rightarrow 0\,, & 
    \sum_{k>m} a_k^p &\rightarrow 0\,,
  \ealnn
  as $m\rightarrow\infty$, for any fixed $n$ the sequence $(R_{n,m}: m > n)$ is Cauchy in $L^p$.
  It thus has an $L^p$ limit $R_{n,\infty}$.
  On the other hand, the argument just below \eqref{e:Ta} shows that $(R_{n,m}: m > n)$ converges almost surely, and thus in probability, to $T_{a,>n}$.
  Uniqueness of limits in probability implies $R_{n,\infty} = T_{a,>n}$.
  Hence,
  \[
    \bbE|T_{a,>n}|^p
    \le C_p\lt[
      \lt(\sum_{k>n}a_k^2\rt)^{p/2}
      +\sum_{k>n}a_k^p
    \rt]\,.
  \]
  The assumption \eqref{e:a-asymptotic} implies $a_k \asymp k^{-2/3}$, so there exists $C_{a,p}$ depending only on $a,p$ such that
  \[
    \bbE|T_{a,>n}|^p
    \le C_{a,p} n^{-p/6}\,.
  \]
  Then, recalling \eqref{e:S-an-density} and $\|f_{a,n}\|_\infty\le\|f_{a,3}\|_\infty$,
  \balnn
    \bbE_{\nu_a}\lt[|S_{a,n}|^p\rt]
    = \int |x|^p \tilde f_{a,n}(x)\,\de x
    &\le \fr{\|f_{a,3}\|_\infty}{p_a(0)}
    \int_{\bbR} |x|^p p_{a,>n}(-x)\,\de x \\
    &= \fr{\|f_{a,3}\|_\infty}{p_a(0)} \bbE|T_{a,>n}|^p
    \le C'_{a,p} n^{-p/6}\,,
  \ealnn
  for some $C'_{a,p}$ depending only on $a,p$.
  Finally, set $\eps_n = n^{-\delta}$, for $\delta>0$ small enough that $p(\fr16 - \delta) > 1$.
  Then,
  \[
    \sum_{n\ge 3}\nu_a(|S_{a,n}|>\eps_n)
    \le \sum_{n\ge 3} \eps_n^{-p} \bbE_{\nu_a}\lt[|S_{a,n}|^p\rt]
    \le C'_{a,p} \sum_{n\ge 3} n^{-p(1/6-\delta)}
    <\infty\,.
  \]
  The Borel--Cantelli lemma then implies $S_{a,n} \rightarrow 0$ almost surely.
\end{proof}

\subsection{Square integrability}
\label{ss:limit-square-integrability}

In this subsection, we prove Proposition~\ref{p:limit-distribution-square-integrable-deterministic}.
The main ingredient is the following second-moment bound for $\nu_a$, proved using the same Fourier representation as in \S\ref{ss:canonical-pin}.
Recall $T_a$ defined in \eqref{e:Ta}, and its characteristic function $\Phi_a$ defined in \eqref{e:Phi}.
\begin{lem}\label{l:pin-moments}
  Let $a$ be a decreasing positive sequence satisfying \eqref{e:a-asymptotic}.
  Then,
  \beq\label{e:pin-moment-formula}
    \bbE_{\nu_a} \lt[g_k^2\rt]
    =\fr{
      \int_{\bbR}
      \Phi_a(u)(1-2iua_k)^{-1}\,\de u
    }{
      \int_{\bbR}\Phi_a(u)\,\de u
    }\,.
  \eeq
  \beq\label{e:pin-moment-bound}
    \sup_{k\ge 1}
    \bbE_{\nu_a}\lt[g_k^2\rt] < \infty\,.
  \eeq
  Moreover, $\nu_a$ is invariant under changing the sign of any individual $g_k$.
\end{lem}
\begin{proof}
  Recall that $\varphi(x) = (2\pi)^{-1/2} \exp(-x^2/2)$ denotes the standard Gaussian density.
  Further recall the density $p_{a,>k}$ defined in Lemma~\ref{l:constraint-densities}. 
  Define
  \[
    q_{a,k}(x)=
    \int_{\bbR^k} z_k^2
    \lt(\prod_{j=1}^k \varphi(z_j)\rt)
    p_{a,>k}\lt(
      x-\sum_{j=1}^k a_j(z_j^2-1)
    \rt)\,\de z\,.
  \]
  Conditioning on the first $k$ coordinates shows that $q_{a,k}$ is the density of the finite measure
  $B\mapsto\bbE[g_k^2\bbone\{T_a\in B\}]$.
  Its continuity follows by dominated convergence from the boundedness and continuity of $p_{a,>k}$ in Lemma~\ref{l:constraint-densities}.
  By the definitions \eqref{e:canonical-pin-marginal-s}, \eqref{e:canonical-pin-definition},
  \[
    \bbE_{\nu_a}\lt[ g_k^2\rt]
    = \fr{q_{a,k}(0)}{p_a(0)}\,.
  \]
  We will evaluate this ratio by Fourier inversion.
  A direct calculation shows
  \[
    \bbE\lt[g_k^2e^{iua_k(g_k^2-1)}\rt]
    = e^{-iua_k}(1-2iua_k)^{-3/2}\,,
  \]
  where we use the principal power.
  Thus the Fourier transform of $q_{a,k}$ is
  \beq\label{e:q-fourier-transform}
    \bbE\lt[g_k^2 e^{iuT_a}\rt]
    = \bbE\lt[g_k^2 e^{iua_k(g_k^2-1)}\rt] \prod_{j\neq k} \bbE\lt[e^{iua_j(g_j^2-1)}\rt]
    = \Phi_a(u) (1 - 2iua_k)^{-1}\,.
  \eeq
  Here the product is justified similarly as in Lemma~\ref{l:constraint-densities}.
  The right-hand side of \eqref{e:q-fourier-transform} and $\Phi_a(u)$ are both integrable by \eqref{e:Phi-decay}.
  Fourier inversion at $0$ yields
  \balnn
    2\pi q_{a,k}(0) &= \int_{\bbR} \Phi_a(u) (1 - 2iua_k)^{-1}\,\de u\,, \\
    2\pi p_a(0) &= \int_{\bbR} \Phi_a(u) \,\de u\,.
  \ealnn
  Dividing these equations proves \eqref{e:pin-moment-formula}.
  Since $|(1-2iua_k)^{-1}|\le 1$, we also have that
  \[
    \bbE_{\nu_a} \lt[g_k^2\rt]
    \le \fr{\int_{\bbR}|\Phi_a(u)|\,\de u}{2\pi p_a(0)}
  \]
  is bounded uniformly in $k$.
  This proves \eqref{e:pin-moment-bound}.
  Finally, every density \eqref{e:canonical-pin-marginal-s} at $s=0$ depends on each $g_k$ only through $g_k^2$.
  Thus every finite-dimensional marginal, and hence $\nu_a$, is invariant under changing the sign of $g_k$.
\end{proof}

\begin{proof}[Proof of
Proposition~\ref{p:limit-distribution-square-integrable-deterministic}]
  Let $g^{(1)},g^{(2)}$ be independent samples from $\nu_a$, and set
  \[
    m_k(a) = \bbE_{\nu_a}\lt[g_k^2\rt].
  \]
  By the conclusion \eqref{e:pin-moment-bound} of Lemma~\ref{l:pin-moments}, there exists $C = C(a) > 0$ such that $\sup_{k\ge 1} m_k(a) \le C$.

  Coordinate sign symmetry and Lemma~\ref{l:pin-moments} imply that $\bbE_{\nu_a}[g_jg_k]=0$ for $j\neq k$.
  Therefore, for any positive integers $m\le n$,
  \beq\label{e:bilinear-tail-isometry}
    \bbE_{\nu_a\otimes\nu_a} \lt|
      \sum_{k=m}^n a_kg_k^{(1)}g_k^{(2)}
    \rt|^2
    = \sum_{k=m}^n a_k^2 m_k(a)^2
    \le C^2 \sum_{k=m}^n a_k^2\,.
  \eeq
  The assumption \eqref{e:a-asymptotic} implies $\sum_k a_k^2 < \infty$.
  Thus the series defining $Q_a$ in Definition~\ref{d:limit-distribution} converges in $L^2$, and
  \[
    \bbE \lt[Q_a^2\rt]
    = \sum_{k\ge1} a_k^2 m_k(a)^2
    \le C^2 \sum_{k\ge 1}a_k^2
    <\infty\,. \qedhere
  \]
\end{proof}

\subsection{Measurability and finite-dimensional pins}
\label{ss:airy-regularity}

The goal of this subsection is to record two facts needed in \S\ref{s:overlap}: the Borel dependence of the limiting pinned overlap law $\fP_a$ on $a$, and a coarea identity relating finite pins to weighted spherical measures.

\begin{lem}\label{l:parameter-measurability}
  Let $\cA \subseteq \ell^2$ consist of the positive nonincreasing sequences $a = (a_k)_{k\ge 1}$ such that \eqref{e:a-asymptotic} holds.
  Then $\cA$ is a Borel subset of $\ell^2$.
  The map
  \[
    a\mapsto\fP_a
  \]
  is Borel from $\cA$ into $(\cP_2(\bbR),\bbW_2)$.
  Consequently,
  \[
    a\mapsto V_a \equiv \|\fP_a\|_{L^2}^2
  \]
  is Borel from $\cA$ into $\bbR$.
\end{lem}

\begin{proof}
  The coordinate maps are continuous on $\ell^2$, and
  \[
    \cA
    =
    \bigcap_{k\ge1}
    \lt(
      \{a\in\ell^2:a_k>0\}
      \cap
      \{a\in\ell^2:a_k\ge a_{k+1}\}
    \rt)
    \cap
    \bigcap_{r\in\bbQ_{>0}}
    \bigcup_{m\ge1}
    \bigcap_{k\ge m}
    \{a\in\ell^2:|a_kt_k-1|<r\}.
  \]
  Thus $\cA$ is Borel, and in particular is a standard Borel space.

  Recall that $\Phi_a$ is the characteristic function of $T_a$, which in the proof of Lemma~\ref{l:constraint-densities} we showed equals the infinite product \eqref{e:constraint-characteristic}.
  Fix $n\ge1$, and write $\Phi_{a,>n}$ for the characteristic function of $T_{a,>n}$.
  By the same proof, $\Phi_{a,>n}$ equals the infinite product \eqref{e:constraint-characteristic} restricted to terms $k>n$, and satisfies the decay rate \eqref{e:Phi-decay}.
  On $\cA\times\bbR$, the products $\Phi_a(u)$ and $\Phi_{a,>n}(u)$ are pointwise limits of finite products and hence are jointly Borel in $(a,u)$.
  The decay rate \eqref{e:Phi-decay} also implies that both are absolutely integrable.
  The parameter-integral theorem and Fourier inversion then show that
  \balnn
    p_a(x)
    &=
    \fr1{2\pi}\int_{\bbR}e^{-iux}\Phi_a(u)\,\de u\,, &
    p_{a,>n}(x)
    &=
    \fr1{2\pi}\int_{\bbR}e^{-iux}\Phi_{a,>n}(u)\,\de u
  \ealnn
  are jointly Borel in $(a,x)$.
  Lemma~\ref{l:constraint-densities} gives $p_a(0)>0$.
  Specializing \eqref{e:canonical-pin-marginal-s} to $s=0$, define
  \[
    \nu_a^{(n)}(\de z_1 \cdots \de z_n)
    =
    \fr{1}{p_a(0)}
    \lt(\prod_{k=1}^n\varphi(z_k)\,\de z_k\rt)
    p_{a,>n}\lt(-\sum_{k=1}^na_k(z_k^2-1)\rt).
  \]
  The preceding joint Borel properties show that $a\mapsto\nu_a^{(n)}$ is a probability kernel from $\cA$ to $\bbR^n$.
  Let $\fP_a^{(n)}$ be the law, under $\nu_a^{(n)}\otimes\nu_a^{(n)}$, of
  \[
    Q_a^{(n)}(z,z')=\sum_{k=1}^na_kz_kz_k'.
  \]
  The map $(a,z,z') \mapsto Q_a^{(n)}(z,z')$ is jointly Borel, and the product of two probability kernels is a probability kernel, so $a\mapsto\fP_a^{(n)}$ is weakly Borel as a map from $\cA$ into $(\cP(\bbR), \text{weak})$.
  Moreover, since $\nu_a^{(n)}$ is the first $n$-coordinate marginal of $\nu_a$, \eqref{e:bilinear-tail-isometry} from the proof of Proposition~\ref{p:limit-distribution-square-integrable-deterministic} shows that $\fP_a^{(n)}\in\cP_2(\bbR)$.
  The inclusion
  \[
    (\cP_2(\bbR),\bbW_2)
    \rightarrow
    (\cP(\bbR),\text{weak})
  \]
  is a continuous injection between Polish spaces.
  By the Lusin--Souslin theorem, its image is Borel and its inverse on that image is Borel.
  It follows that $a\mapsto\fP_a^{(n)}$ is Borel as a map from $\cA$ into $(\cP_2(\bbR),\bbW_2)$.

  Finally, we sample independent $g^{(1)}, g^{(2)} \sim \nu_a$ and couple $Q_a^{(n)}$ with $Q_a$ by
  \balnn
    Q^{(n)}_a &= \sum_{k=1}^n a_k g^{(1)}_k g^{(2)}_k\,, &
    Q_a &= \sum_{k\ge 1} a_k g^{(1)}_k g^{(2)}_k\,.
  \ealnn
  Then \eqref{e:bilinear-tail-isometry} gives
  \[
    \bbW_2^2(\fP_a^{(n)},\fP_a)
    \le \bbE_{\nu_a\otimes\nu_a}|Q_a^{(n)}-Q_a|^2
    = \sum_{k>n}a_k^2m_k(a)^2
    \rightarrow 0\,.
  \]
  Thus $a\mapsto\fP_a$ is the pointwise limit of Borel maps into the Polish space $(\cP_2(\bbR),\bbW_2)$ and is therefore Borel.
  Finally,
  \[
    V_a
    = \|\fP_a\|_{L^2}^2
    = \bbW_2^2(\fP_a,\delta_0)\,,
  \]
  so $a\mapsto V_a$ is Borel by continuity of the last expression in $\bbW_2$.
\end{proof}

For a finite positive vector $c=(c_1,\ldots,c_M)$ with $M\ge3$, let
\[
  F_c(z)=\sum_{k=1}^Mc_k(z_k^2-1).
\]
Zero is a regular value of $F_c$: if $\nabla F_c(z)=0$, then $z=0$, and $F_c(0)=-\sum_kc_k<0$.

Let $p_c$ denote the density of $F_c(g)$ for $g \sim \cN(0, I_M)$.
By Lemma~\ref{l:constraint-densities-positive}, $p_c$ is bounded and continuous, with $p_c(0) > 0$.
Let $\cH^{M-1}$ denote $(M-1)$-dimensional Hausdorff measure.
Then the coarea formula shows that
\[
  \tilde p_c(t)
  = \int_{F_c^{-1}(t)}
  \fr{1}{\|\nabla F_c(z)\|}
  \prod_{k=1}^M \varphi(z_k)
  \,\de\cH^{M-1}(z)
\]
is a version of the density of $F_c(g)$.
Since $0$ is a regular value of $F_c$ and $F_c^{-1}(0)$ is compact, $\tilde p_c$ is continuous in a neighborhood of $0$.
On this neighborhood, $p_c$ and $\tilde p_c$ are continuous versions of the same density, and hence $p_c(0)=\tilde p_c(0)$.

We then define the finite-dimensional pinned measure on $F_c^{-1}(0)$ by
\beq\label{e:finite-coarea-pin}
  \nu_c(\de z)
  = \fr1{p_c(0)\|\nabla F_c(z)\|}
  \prod_{k=1}^M \varphi(z_k)
  \,\de\cH^{M-1}(z)\,.
\eeq
Set $m_k(c) = \bbE_{\nu_c}[g_k^2]$ and $V_c = \sum_{k=1}^M c_k^2 m_k(c)^2$.
Only $M\ge3$ will be used below.

\begin{lem}[Diagonal coarea change of variables]
  \label{l:diagonal-coarea}
  Let $c=(c_1,\ldots,c_M)$ be positive, where $M\ge3$, and put $\alpha_c=\sum_{k=1}^Mc_k$.
  Under the map
  \[
    y_k=c_k^{1/2}z_k,
  \]
  write $\widetilde\nu_c$ for the pushforward of $\nu_c$, and let $\omega_r$ denote normalized surface measure on $r\bbS^{M-1}$.
  Then
  \beq\label{e:coarea-spherical-density}
    \widetilde\nu_c(\de y)
    =
    \fr1{Z_c}
    \exp\lt(-\fr12\sum_{k=1}^M\fr{y_k^2}{c_k}\rt)
    \,\de\omega_{\sqrt{\alpha_c}}(y),
  \eeq
  where $Z_c$ is the normalizing constant.
\end{lem}

\begin{proof}
  Let $L=\diag(c_1^{1/2},\ldots,c_M^{1/2})$, so $y=Lz$ maps $F_c^{-1}(0)$ onto $\sqrt{\alpha_c}\,\bbS^{M-1}$.
  For the unit normal $n_z=\nabla F_c(z)/\|\nabla F_c(z)\|$, the tangential Jacobian is $JL(z)=|\det L|\,\|L^{-1}n_z\|$.
  Since $\nabla F_c(z)=2Ly$, we have $n_z=Ly/\|Ly\|$.
  Then, for $z\in F_c^{-1}(0)$,
  \[
    \|\nabla F_c(z)\|JL(z)
    =2|\det L|\,\|y\|
    =2\sqrt{\alpha_c}\,|\det L|\,.
  \]
  Thus the coarea factor $1/\|\nabla F_c(z)\|$ in \eqref{e:finite-coarea-pin} and the surface Jacobian combine to a constant under the area formula.
  The only remaining $y$-dependent factor is
  \[
    \exp\lt(-\fr12\|L^{-1}y\|^2\rt)
    =
    \exp\lt(-\fr12\sum_{k=1}^M\fr{y_k^2}{c_k}\rt)\,.
  \]
  Normalization gives \eqref{e:coarea-spherical-density}.
\end{proof}

\section{Sphere to cube comparison principle}
\label{s:sphere-to-cube}

In this section we prove Theorem~\ref{t:sphere-to-cube}. The main new input is the following proposition.
\begin{ppn}\label{p:single-f-test}
There exists a universal constant $t_0>0$ such that, for every fixed $t\in(-t_0,t_0)$,
    \[
    \la \exp(tN^{1/3}R_{1,2})\ra - \la \exp(tN^{1/3}R_{1,2})\ra^\sph \rightarrow 0    \]
in probability as $N\to\infty$, with respect to $\bW\sim\GOE(N)$.
\end{ppn}

The proofs of Proposition~\ref{p:single-f-test} and Theorem~\ref{t:sphere-to-cube} both rely crucially on the exponential tail estimates from \cite{du2026fluctuations} that we now recall.

\begin{ppn}[{\cite[Theorem 1.4(a), Corollary 1.7(b)]{du2026fluctuations}}]
  \label{p:exponential-moment-bound}
    There exists a universal constant $c>0$ such that for any $N\ge 2$,
    \begin{align*}
    &\mathbb{E}\langle \exp(cN^{1/3}|R_{1,2}|)\rangle \le 2\,, &\mathbb{E}\langle \exp(cN^{1/3}|R_{1,2}|)\rangle^\sph \le 2\,.
    \end{align*}
    Consequently, for any $x>0$, it holds that
    \begin{align*}
        &\mathbb{E}\langle \bbone\{N^{1/3}|R_{1,2}|>x\}\rangle\le 2\exp(-cx)\,,&\mathbb{E}\langle \bbone\{N^{1/3}|R_{1,2}|>x\}\rangle^\sph\le 2\exp(-cx)\,.
    \end{align*}
\end{ppn}

Assuming Propositions~\ref{p:single-f-test}-\ref{p:exponential-moment-bound}, we now proceed to prove Theorem~\ref{t:sphere-to-cube}. For simplicity,
define $\zeta_N,\zeta_N^\sph\in \mathcal P_2(\mathcal P_2(\bbR))$ as the respective laws of the random probability measures $\la \delta(N^{1/3}R_{1,2})\ra$ and $\la \delta(N^{1/3}R_{1,2})\ra^\sph$ over $\bW\sim \GOE(N)$.
We will prove $\cW_2(\zeta_N,\zeta_N^\sph)\to 0$ by a compactness argument, and we start with a lemma that allows us to take a subsequential limit.

\begin{lem}\label{l:relative-compactness}
    Both $\{\zeta_N\}$ and $\{\zeta_N^\sph\}$ are relatively compact with respect to $(\mathcal P_2(\mathcal P_2(\bbR)),\cW_2)$.
\end{lem}

\begin{proof}
Recall that $\cW_2$ is the Wasserstein-2 distance on the Polish space $(\mathcal P_2(\bbR),\bbW_2)$. By Prokhorov's theorem and the characterization of Wasserstein convergence (see, e.g., \cite[Theorem 6.9]{Villani2009}), relative compactness is equivalent to tightness together with uniform integrability of the second moments. We treat $\{\zeta_N\}$; the proof for $\{\zeta_N^\sph\}$ is identical.

By Proposition~\ref{p:exponential-moment-bound}, for any $M>0$ and $X$ a sample from $\mu_N$,
\[
\sup_{N}\mathbb{E}_{\mu_N\sim \zeta_N}\left[\mu_N(|X|\ge M)\right]\le 2\exp(-cM)\,.
\]
This implies that
\[
\sup_N\mathbb{E}_{\mu_N\sim \zeta_N}\left[\int x^4\mu_N(\de x)\right]\le \int_0^\infty 8M^3\exp(-cM)\de M<\infty\,.
\]
Consequently via Markov's inequality, as $K\to \infty$, uniformly over $N$ it holds
\[
\zeta_N(\mu_N:\bbW_4(\mu_N,\delta_0)>K)=\zeta_N\left(\mu_N:\int x^4\mu_N(\de x)>K^4\right)\le O(K^{-4})\to 0.
\]
Since the closed $\bbW_4$-ball $\oB_{\bbW_4}(\delta_0,K)$ is compact in $\bbW_2$ for any $K>0$, tightness of $\{\zeta_N\}$ follows.

For uniform integrability, we need to show that as $K\to\infty$,
\[
\sup_N\int_{\{\bbW_2(\mu_N,\delta_0)>K\}}\bbW_2(\mu_N,\delta_0)^2 \zeta_N(\de \mu_N)=\sup_N\int_{\mu_N:\int x^2 \mu_N(\de x)>K^2}\int x^2\mu_N(\de x)\zeta_N(\de\mu_N)\to 0\,.
\]
Indeed, by Chebyshev and Cauchy--Schwarz, uniformly over $N$ it holds
\balnn
\int_{\mu_N:\int x^2 \mu_N(\de x)>K^2}\int x^2\mu_N(\de x)\zeta_N(\de\mu_N)&\le K^{-2}\int \left(\int x^2\mu_N(\de x)\right)^2\zeta_N(\de\mu_N)\\
&\le K^{-2}\int \int x^4\mu_N(\de x)\zeta_N(\de\mu_N)=O(K^{-2})\,,
\ealnn
which tends to $0$ as $K\to\infty$. This proves the lemma.
\end{proof}

We are now ready to give the proof of Theorem~\ref{t:sphere-to-cube}. 

\begin{proof}[Proof of Theorem~\ref{t:sphere-to-cube}]
By Lemma~\ref{l:relative-compactness}, every subsequence of
$\{(\zeta_N,\zeta_N^\sph)\}$ admits a further subsequence, denoted by
$\{N_j\}$, such that
\[
\zeta_{N_j}\stackrel{\cW_2}{\longrightarrow}\zeta,
\qquad
\zeta_{N_j}^\sph\stackrel{\cW_2}{\longrightarrow}\zeta^\sph.
\]
We claim that $\zeta=\zeta^\sph$. Once this is established, the triangle
inequality gives
\[
\cW_2(\zeta_{N_j},\zeta_{N_j}^\sph)
\le
\cW_2(\zeta_{N_j},\zeta)
+
\cW_2(\zeta_{N_j}^\sph,\zeta^\sph)
\rightarrow 0.
\]
Since this holds along a further subsequence of every subsequence, the desired
convergence follows.

To prove the claim, let
$
\pi_N\in
\mathcal P_2\bigl(\mathcal P_2(\bbR)\times\mathcal P_2(\bbR)\bigr)
$
be the coupling of $\zeta_N$ and $\zeta_N^\sph$ obtained by using a common
$\bW\sim\GOE(N)$. By Lemma~\ref{l:relative-compactness}, both marginal
families are tight in $(\mathcal P_2(\bbR),\bbW_2)$, and hence $\{\pi_N\}$ is
tight in
$
\bigl(\mathcal P_2(\bbR)\times\mathcal P_2(\bbR),
\bbW_2\times\bbW_2\bigr).
$
This is a Polish space, so Prokhorov's theorem allows us, after passing to a
further subsequence, to assume that
\[
\pi_{N_j}
=
\Law((\mu_{N_j},\mu_{N_j}^\sph))
\Rightarrow
\pi
=
\Law((\mu,\mu^\sph))
\]
for some
$
\pi\in\mathcal P\bigl(\mathcal P_2(\bbR)\times\mathcal P_2(\bbR)\bigr).
$
Since $\cW_2$-convergence implies weak convergence and coordinate projections
are continuous, the marginals of $\pi$ are $\zeta$ and $\zeta^\sph$.
Therefore, it remains to show that $\pi$ is supported on the diagonal.

Fix $t\in\bbR$ such that
$
|t|<\min\{t_0,c\},
$
where $t_0$ and $c$ are the constants in
Propositions~\ref{p:single-f-test} and~\ref{p:exponential-moment-bound},
respectively. Proposition~\ref{p:single-f-test} gives
\begin{equation}\label{e:converge-in-probability}
\int e^{tx}\mu_{N_j}(\de x)
-
\int e^{tx}\mu_{N_j}^\sph(\de x)
=
\langle \exp(tN_j^{1/3}R_{1,2})\rangle
-
\langle \exp(tN_j^{1/3}R_{1,2})\rangle^\sph
\rightarrow0
\end{equation}
in probability.

Let $\delta>0$. By Proposition~\ref{p:exponential-moment-bound}, we may choose
$K=K(\delta)$ sufficiently large that
\begin{align}\label{e:tail-delta}
\mathbb E_{\mu_{N_j}\sim\zeta_{N_j}}
\int e^{tx}\bbone\{|x|>K\}\mu_{N_j}(\de x)
&<\delta^2,
&
\mathbb E_{\mu_{N_j}^\sph\sim\zeta_{N_j}^\sph}
\int e^{tx}\bbone\{|x|>K\}\mu_{N_j}^\sph(\de x)
&<\delta^2
\end{align}
uniformly in $j$. Choose $\eta\in C_c^\infty(\bbR)$ such that
$0\le\eta\le1$ and $\eta=1$ on $[-K,K]$, and define
\[
F(\nu,\nu')
=
\int e^{tx}\eta(x)\nu(\de x)
-
\int e^{tx}\eta(x)\nu'(\de x).
\]
Note that by Markov's inequality,
\begin{align*}
\pi_{N_j}\bigl(|F(\mu_{N_j},\mu_{N_j}^\sph)|>3\delta\bigr)
\le{}&
\pi_{N_j}\left(\left|\int e^{tx}\mu_{N_j}(\de x)
-\int e^{tx}\mu_{N_j}^\sph(\de x)\right|>\delta\right)\\
&+\frac{1}{\delta}\mathbb E_{\mu_{N_j}\sim \zeta_{N_j}}\int e^{tx}(1-\eta(x))\,\mu_{N_j}(\de x)\\
&+\frac{1}{\delta}\mathbb E_{\mu_{N_j}^\sph\sim \zeta_{N_j}^\sph}\int e^{tx}(1-\eta(x))\,\mu_{N_j}^\sph(\de x).
\end{align*}
Therefore, by \eqref{e:converge-in-probability}, \eqref{e:tail-delta}, 
\[
\limsup_{j\to\infty}
\pi_{N_j}\bigl(|F(\mu_{N_j},\mu_{N_j}^\sph)|>3\delta\bigr)
\le 2\delta.
\]

Since $e^{tx}\eta(x)$ is bounded and continuous, $F$ is continuous with
respect to the product $\bbW_2$-topology. Indeed, if $(\mu_n,\nu_n)\to(\mu,\nu)$ in the product $\bbW_2$-topology, then
$\mu_n\Rightarrow\mu$ and $\nu_n\Rightarrow\nu$. Since $e^{tx}\eta(x)$ is
bounded and continuous,
$
F(\mu_n,\nu_n)\rightarrow F(\mu,\nu).
$ Thus, by the continuous mapping
theorem,
\[
F(\mu_{N_j},\mu_{N_j}^\sph)
\Rightarrow
F(\mu,\mu^\sph).
\]
The Portmanteau theorem then gives
\begin{equation}\label{e:limit-probability}
\pi\bigl(|F(\mu,\mu^\sph)|>3\delta\bigr)
\le
\liminf_{j\to\infty}
\pi_{N_j}\bigl(|F(\mu_{N_j},\mu_{N_j}^\sph)|>3\delta\bigr)
\le 2\delta.
\end{equation}

The exponential-tail estimates also pass to the limiting marginals. Indeed,
for every $x>0$, the maps
\[
\nu\mapsto\nu((x,\infty))
\qquad\text{and}\qquad
\nu\mapsto\nu((-\infty,-x))
\]
are lower semicontinuous under weak convergence. Hence the same uniform
exponential-tail bounds hold under $\zeta$ and $\zeta^\sph$. Since $|t|<c$,
after increasing $K$ if necessary, we therefore have
\[
\mathbb E_{\mu\sim \zeta}
\int e^{tx}(1-\eta(x))\mu(\de x)
<\delta^2,
\qquad
\mathbb E_{\mu^\sph\sim \zeta^\sph}
\int e^{tx}(1-\eta(x))\mu^\sph(\de x)
<\delta^2.
\]
In particular, the corresponding Laplace transforms are finite
$\pi$-almost surely for $|t|<c$. Combining these estimates with
\eqref{e:limit-probability} and applying Markov's inequality yields
\[
\pi\left(
\left|
\int e^{tx}\mu(\de x)
-
\int e^{tx}\mu^\sph(\de x)
\right|>5\delta
\right)
\le 4\delta.
\]
Since $\delta>0$ is arbitrary, this implies
\[
\int e^{tx}\mu(\de x)
=
\int e^{tx}\mu^\sph(\de x)
\quad
\pi\text{-almost surely}.
\]

Let $D$ be a countable dense subset of
$
\bigl(-\min\{t_0,c\},\min\{t_0,c\}\bigr).
$
Taking the intersection of the preceding probability-one events over $t\in D$,
we obtain that, for $\pi$-almost every $(\mu,\mu^\sph)$,
\[
\int e^{tx}\mu(\de x)
=
\int e^{tx}\mu^\sph(\de x)
\qquad
\text{for every }t\in D.
\]
For each such pair, both Laplace transforms are finite and continuous on a
neighborhood of the origin. Since they agree on the dense set $D$, they agree
throughout that neighborhood. The uniqueness of the Laplace transform
therefore implies that
$
\mu=\mu^\sph.
$
Thus $\pi$ is supported on the diagonal, so its two marginals coincide:
$\zeta=\zeta^\sph$. This completes the proof.
\end{proof}

The rest of this section is devoted to proving Proposition~\ref{p:single-f-test}. Throughout we abbreviate
\begin{align*}
&Z_N:=Z_{N,\beta=1}=\frac{1}{2^N}\sum_{\bx\in \Sigma_N}e^{H_N(\bx)}\,, &Z_N^\sph:=Z_{N,\beta=1}^{\sph}=\int_{S_N}e^{H_N(\bx)}\de\nu_N(\bx)\,,
\end{align*}
and denote $X_N=Z_N/Z_N^\sph$. It was shown in \cite[Theorem 1.6]{du2026fluctuations} that $\mathbb{E}[(X_N-1)^2]\lesssim N^{-1/3}$ and thus $X_N\to 1$ in probability as $N\to\infty$.
Moreover, for simplicity, for \(t\in\mathbb{R}\), we write
\[
L_N(t)
    := \left\langle \exp\bigl(tN^{1/3}R_{1,2}\bigr)\right\rangle,
\qquad
L_N^{\sph}(t)
    := \left\langle \exp\bigl(tN^{1/3}R_{1,2}\bigr)\right\rangle^{\sph}.
\]
We will show that for small enough $|t|$,
\begin{equation}\label{e:weighted-convergence-in-probability}
X_N^2L_N(t)-L_N^\sph(t)\rightarrow 0
\end{equation}
in probability as $N\to\infty$. This together with the facts that $X_N\to 1$ in probability and $L_N(t),L_N^\sph(t)$ are bounded in probability implies $L_N(t)-L_N^\sph(t)\to 0$ in probability for all small enough $t$, as desired in Proposition~\ref{p:single-f-test}.

Towards proving \eqref{e:weighted-convergence-in-probability}, we first
establish an annealed comparison, and then upgrade it to the quenched sense
via a vector-embedding argument. These two main steps are presented in \S\ref{ss:annealed-convergence} and \S\ref{ss:quenched-convergence} below, and the proof of Proposition~\ref{p:single-f-test} is completed in \S\ref{ss:proof-of-single-test}. 

\subsection{Annealed convergence}\label{ss:annealed-convergence}

We first prove the annealed version of
\eqref{e:weighted-convergence-in-probability}.

\begin{lem}\label{l:annealed-convergence}
There exists a universal constant \(t_0>0\) such that, for every
\(t\in(-t_0,t_0)\),
\[
\mathbb{E}\bigl[X_N^2L_N(t)-L_N^\sph(t)\bigr]
\rightarrow 0
\qquad\text{as }N\to\infty.
\]
\end{lem}

\begin{proof}
We follow the sphere-to-cube comparison argument from
\cite[\S3]{du2026fluctuations}. Let
\[
\mathcal Q_N
    :=\left\{-1,-1+\frac{2}{N},\ldots,1-\frac{2}{N},1\right\}
\]
be the set of possible overlaps between two points in \(\Sigma_N\), and
write
\[
p_N(q)
    :=2^{-N}\binom{N}{N(1+q)/2},
    \qquad q\in\mathcal Q_N,
\]
for the overlap distribution of two independent uniform points in
\(\Sigma_N\). Likewise, let
\[
\rho_N(q)
    :=\frac{\Gamma(N/2)}
            {\sqrt{\pi}\Gamma((N-1)/2)}
      (1-q^2)^{(N-3)/2},
    \qquad q\in[-1,1],
\]
be the overlap density of two independent uniform points in \(S_N\).

For \(q\in[-1,1]\), choose any
\(\bx,\by\in S_N\) satisfying
\(R(\bx,\by)=q\), and define
\[
J_N(q)
    :=\mathbb{E}\left[
        \frac{
        e^{H_N(\bx)+H_N(\by)}
        }{
        (Z_N^\sph)^2
        }
      \right].
\]
By rotational invariance, \(J_N(q)\) depends on
\(\bx,\by\) only through their overlap \(q\). The reweighted
overlap identities of \cite[Lemma~2.1]{du2026fluctuations} give
\begin{align}
\mathbb{E}\bigl[X_N^2L_N(t)\bigr]
    &=
      \sum_{q\in\mathcal Q_N}
      p_N(q)J_N(q)e^{tN^{1/3}q},
      \label{e:weighted-cube-identity}
      \\
\mathbb{E}\bigl[L_N^\sph(t)\bigr]
    &=
      \int_{-1}^{1}
      \rho_N(q)J_N(q)e^{tN^{1/3}q}\,\de q.
      \label{e:weighted-spherical-identity}
\end{align}

We recall two ingredients from the sphere-to-cube comparison. First,
\(J_N\) admits the factorization
\begin{equation}\label{e:J-factorization}
J_N(q)=e^{Nq^2/2}e^{N/2}K_N(q),
\end{equation}
where \(K_N\) satisfies
\begin{equation}\label{e:K-local-regularity}
\left|
\log K_N(q)-\log K_N(\widetilde q)
\right|
\leq
N\bigl(|q|+N^{-1}\bigr)|q-\widetilde q|
\end{equation}
whenever \(q,\widetilde q\in[-1,1]\) and
\(|q-\widetilde q|\leq N^{-1}\); see
\cite[Corollary~3.4]{du2026fluctuations}. Second, the local density
comparison in \cite[Lemma~3.9]{du2026fluctuations} states that, uniformly
over \(q\in\mathcal Q_N\) with \(|q|=o(N^{-1/4})\) and
\(|\widetilde q-q|\leq N^{-1}\),
\begin{equation}\label{e:local-density-comparison}
p_N(q)e^{Nq^2/2}
=
\frac{2}{N}
\rho_N(\widetilde q)e^{N\widetilde q^2/2}
\left(
1+O\bigl(N^{-1}+q^2+Nq^4\bigr)
\right).
\end{equation}

By \cite[Propositions~3.1--3.2]{du2026fluctuations} and
Proposition~\ref{p:exponential-moment-bound}, there exist universal constants
\(c_0,C>0\) such that
\begin{equation}\label{e:annealed-exponential-localization}
\sup_{N\geq1}
\left\{
\mathbb{E}\left[
X_N^2
\left\langle
e^{c_0N^{1/3}|R_{1,2}|}
\right\rangle
\right]
+
\mathbb{E}\left[
\left\langle
e^{c_0N^{1/3}|R_{1,2}|}
\right\rangle^\sph
\right]
\right\}
\leq C.
\end{equation}
Set \(t_0:=c_0/2\), and fix \(t\in(-t_0,t_0)\).

For \(M>0\), define
\[
A_{N,M}
    :=\mathcal Q_N\cap
      \bigl[-MN^{-1/3},MN^{-1/3}\bigr],
\qquad
I_{N,M}
    :=\bigcup_{q\in A_{N,M}}
      \bigl[q-N^{-1},q+N^{-1}\bigr].
\]
We first compare the two expressions in
\eqref{e:weighted-cube-identity}--\eqref{e:weighted-spherical-identity}
on this central region. Fix \(M\). Uniformly over
\(q\in A_{N,M}\) and \(|\widetilde q-q|\leq N^{-1}\), we have
\[
N^{-1}+q^2+Nq^4=O_M(N^{-1/3})=o(1),
\]
and \eqref{e:K-local-regularity} gives
\[
\left|
\log K_N(q)-\log K_N(\widetilde q)
\right|
\leq
N\bigl(MN^{-1/3}+N^{-1}\bigr)N^{-1}
=o(1).
\]
Moreover,
\[
\left|
tN^{1/3}(q-\widetilde q)
\right|
\leq |t|N^{-2/3}=o(1).
\]
Combining these estimates with
\eqref{e:J-factorization} and
\eqref{e:local-density-comparison}, we obtain, uniformly over
\(q\in A_{N,M}\),
\begin{align}
p_N(q)J_N(q)e^{tN^{1/3}q}
&=
\bigl(1+o(1)\bigr)
\int_{q-N^{-1}}^{q+N^{-1}}
\rho_N(\widetilde q)J_N(\widetilde q)
e^{tN^{1/3}\widetilde q}\,\de\widetilde q.
\label{e:weighted-local-comparison}
\end{align}
The \(o(1)\) here may depend on \(M\) and \(t\), but is uniform in
\(q\in A_{N,M}\). Summing
\eqref{e:weighted-local-comparison} over \(q\in A_{N,M}\), and using
\eqref{e:annealed-exponential-localization}, yields
\begin{equation}\label{e:central-weighted-comparison}
\left|
\sum_{q\in A_{N,M}}
p_N(q)J_N(q)e^{tN^{1/3}q}
-
\int_{I_{N,M}}
\rho_N(q)J_N(q)e^{tN^{1/3}q}\,\de q
\right|
=o(1).
\end{equation}

It remains to control the tails. On
\(\{|q|>MN^{-1/3}\}\), we have
\[
e^{tN^{1/3}q}
\leq
e^{|t|N^{1/3}|q|}
\leq
e^{-(c_0-|t|)M}
e^{c_0N^{1/3}|q|}.
\]
Consequently, by
\eqref{e:weighted-cube-identity} and
\eqref{e:annealed-exponential-localization},
\begin{equation}\label{e:cube-weighted-tail}
\sum_{q\in\mathcal Q_N\setminus A_{N,M}}
p_N(q)J_N(q)e^{tN^{1/3}q}
\leq
Ce^{-(c_0-|t|)M}.
\end{equation}
Similarly, every \(q\in[-1,1]\setminus I_{N,M}\) satisfies
\[
|q|\geq MN^{-1/3}-N^{-1}.
\]
It follows from
\eqref{e:weighted-spherical-identity} and
\eqref{e:annealed-exponential-localization} that
\begin{equation}\label{e:spherical-weighted-tail}
\int_{[-1,1]\setminus I_{N,M}}
\rho_N(q)J_N(q)e^{tN^{1/3}q}\,\de q
\leq
C\exp\left(
-(c_0-|t|)
\bigl(M-N^{-2/3}\bigr)
\right).
\end{equation}

Combining
\eqref{e:central-weighted-comparison}--\eqref{e:spherical-weighted-tail},
we obtain
\[
\limsup_{N\to\infty}
\left|
\mathbb{E}\bigl[X_N^2L_N(t)\bigr]
-
\mathbb{E}\bigl[L_N^\sph(t)\bigr]
\right|
\leq
2Ce^{-(c_0-|t|)M}.
\]
Letting \(M\to\infty\) proves the lemma.
\end{proof}

\subsection{Hilbert-space embedding}\label{ss:quenched-convergence}

We next upgrade Lemma~\ref{l:annealed-convergence} to a quenched
convergence statement by means of a Hilbert-space embedding. Since
\(H_N(-\bx)=H_N(\bx)\) and both reference measures are invariant
under \(\bx\mapsto-\bx\), we have
\[
L_N(-t)=L_N(t),
\qquad
L_N^\sph(-t)=L_N^\sph(t).
\]
It therefore suffices to consider \(t\in[0,t_0)\).

For each \(N\geq 1\), let
\[
\mathcal H_N
:=
\bigoplus_{k=0}^{\infty}(\mathbb R^N)^{\otimes k}
\]
be the Hilbert direct sum of the tensor powers of \(\mathbb R^N\), where
\((\mathbb R^N)^{\otimes 0}:=\mathbb R\). On each tensor power, we use the
inner product determined by
\[
\left(
\bx_1\otimes\cdots\otimes\bx_k,\,
\by_1\otimes\cdots\otimes\by_k
\right)
:=
\prod_{j=1}^k
\big(\bx_j,\by_j\big)_{\mathbb R^N}.
\]
Thus, if \(\bu=(\bu_k)_{k\geq0}\) and
\(\bv=(\bv_k)_{k\geq0}\) belong to \(\mathcal H_N\), then
\[
\left(\bu,\bv\right)_{\mathcal H_N}
:=
\sum_{k=0}^{\infty}
\left(\bu_k,\bv_k\right)_{(\mathbb R^N)^{\otimes k}}.
\]

For \(t\geq0\), define the map
\[
\Phi_{N,t}:\mathbb R^N\rightarrow\mathcal H_N,\quad 
\bx\mapsto
\bigoplus_{k=0}^{\infty}
\frac{t^{k/2}N^{-k/3}}{\sqrt{k!}}\,
\bx^{\otimes k},
\]
where \(\bx^{\otimes0}:=1\). For
\(\bx,\by\in S_N\), and in particular for
\(\bx,\by\in\Sigma_N\), this embedding satisfies
\begin{align}
\left(
\Phi_{N,t}(\bx),\Phi_{N,t}(\by)
\right)_{\mathcal H_N}
&=
\sum_{k=0}^{\infty}
\frac{t^kN^{-2k/3}}{k!}
\left(\bx,\by\right)^k
\nonumber\\
&=
\exp(
tN^{-2/3}\left(\bx,\by)
\right)
=
\exp(tN^{1/3}R(\bx,\by)).
\label{e:exponential-kernel-identity}
\end{align}
For \(t\geq 0\), define the Hilbert-space barycenters
\[
\bm_N(t)
    :=\left\langle\Phi_{N,t}(\bx)\right\rangle,
\qquad
\bm_N^\sph(t)
    :=\left\langle\Phi_{N,t}(\bx)\right\rangle^\sph.
\]
The key to our proof is the next lemma. 

\begin{lem}\label{l:hilbert-barycenter-convergence}
For every \(t\in[0,t_0)\), 
\begin{align*}
&\|\bm_N(t)\|_{\mathcal H_N}^2=L_N(t)\,,
&\|\bm^\sph_N(t)\|_{\mathcal H_N}^2=L^\sph_N(t)\,.
\end{align*}
Moreover, 
\[
\mathbb E\left[
\left\|
X_N\bm_N(t)-\bm_N^\sph(t)
\right\|_{\mathcal H_N}^2
\right]\to 0\qquad\text{as }N\to\infty\,.
\]
\end{lem}

\begin{proof}

By \eqref{e:exponential-kernel-identity},
\begin{align*}
\left\|\bm_N(t)\right\|_{\mathcal H_N}^2
&=
\left\langle
\left(
\Phi_{N,t}(\bx^1),
\Phi_{N,t}(\bx^2)
\right)_{\mathcal H_N}
\right\rangle
=
\langle
\exp(tN^{1/3}R_{1,2})
\rangle
=
L_N(t),
\end{align*}
and similarly,
\begin{equation*}
\|\bm_N^\sph(t)\|_{\mathcal H_N}^2
=
L_N^\sph(t).
\end{equation*}
Expanding the square gives
\begin{align*}
\mathbb E\left[
\left\|
X_N\bm_N(t)-\bm_N^\sph(t)
\right\|_{\mathcal H_N}^2
\right]
&=
\mathbb E\left[
X_N^2\left\|\bm_N(t)\right\|_{\mathcal H_N}^2
-2X_N
\left(
\bm_N(t),\bm_N^\sph(t)
\right)_{\mathcal H_N}
+\left\|\bm_N^\sph(t)\right\|_{\mathcal H_N}^2\right]\\
&=\mathbb{E}[X_N^2L_N(t)]+\mathbb{E}[L_N^\sph(t)]-2\mathbb{E}\left[X_N
\left(
\bm_N(t),\bm_N^\sph(t)
\right)_{\mathcal H_N}\right].
\end{align*}
For the mixed term, using
\(X_N=Z_N/Z_N^\sph\) we have
\begin{align*}
&\mathbb E\left[
X_N
\left(
\bm_N(t),\bm_N^\sph(t)
\right)_{\mathcal H_N}
\right]=
\frac{1}{2^N}
\sum_{\bx\in\Sigma_N}
\int_{S_N}
\mathbb E\left[
\frac{
e^{H_N(\bx)+H_N(\by)}
}{
(Z_N^\sph)^2
}
\right]
e^{tN^{1/3}R(\bx,\by)}
\,\de\nu_N(\by).
\end{align*}
For \(\bx\in S_N\), set
\[
F_{N,t}(\bx)
:=
\int_{S_N}
\mathbb E\left[
\frac{
e^{H_N(\bx)+H_N(\by)}
}{
(Z_N^\sph)^2
}
\right]
e^{tN^{1/3}R(\bx,\by)}
\,\de\nu_N(\by).
\]
By rotational invariance of the disorder and of \(\nu_N\),
\[
F_{N,t}(U\bx)=F_{N,t}(\bx)
\qquad
\text{for every }U\in O(N).
\]
Since \(O(N)\) acts transitively on \(S_N\), the function \(F_{N,t}\)
is constant on \(S_N\). Therefore,
\begin{align*}
&\ \frac{1}{2^N}\sum_{\bx\in\Sigma_N}F_{N,t}(\bx)
=
\int_{S_N}F_{N,t}(\bx)\,\de\nu_N(\bx)
\nonumber\\
=&\ 
\mathbb E\left[
\frac{1}{(Z_N^\sph)^2}
\int_{S_N^2}
e^{H_N(\bx)+H_N(\by)}
e^{tN^{1/3}R(\bx,\by)}
\,\de\nu_N(\bx)\de\nu_N(\by)
\right]=
\mathbb E\left[L_N^\sph(t)\right].
\end{align*}
Altogether we obtain
\begin{align*}
\mathbb E\left[
\left\|
X_N\bm_N(t)-\bm_N^\sph(t)
\right\|_{\mathcal H_N}^2
\right]
&=
\mathbb E\left[X_N^2L_N(t)\right]
-\mathbb E\left[L_N^\sph(t)\right].
\end{align*}
The right-hand side converges to zero by
Lemma~\ref{l:annealed-convergence}, which completes the proof.
\end{proof}

\subsection{Quenched convergence}\label{ss:proof-of-single-test}

\begin{proof}[Proof of Proposition~\ref{p:single-f-test}]
Fix \(t\in[0,t_0)\), and abbreviate
\[
\bm_N:=\left\langle\Phi_{N,t}(\bx)\right\rangle,
\qquad
\bm_N^\sph
:=\left\langle\Phi_{N,t}(\bx)\right\rangle^\sph.
\]
By Lemma~\ref{l:hilbert-barycenter-convergence},
\[
X_N^2L_N(t)=\|X_N\bm_N\|_{\mathcal H_N}^2,
\qquad
L_N^\sph(t)=\|\bm_N^\sph\|_{\mathcal H_N}^2.
\]
Using
\[
\bigl|\|\bu\|^2-\|\bv\|^2\bigr|
\leq
\|\bu-\bv\|(\|\bu\|+\|\bv\|),
\]
followed by the Cauchy--Schwarz inequality, we obtain
\begin{align*}
\mathbb E\left[
\left|X_N^2L_N(t)-L_N^\sph(t)\right|
\right]
&\leq
\mathbb E\left[
\left\|X_N\bm_N-\bm_N^\sph\right\|_{\mathcal H_N}
\left(
\|X_N\bm_N\|_{\mathcal H_N}
+\|\bm_N^\sph\|_{\mathcal H_N}
\right)
\right]
\\
&\leq
\left(
\mathbb E
\left\|X_N\bm_N-\bm_N^\sph\right\|_{\mathcal H_N}^2
\right)^{1/2}
\left(
\mathbb E
\left(
\|X_N\bm_N\|_{\mathcal H_N}
+\|\bm_N^\sph\|_{\mathcal H_N}
\right)^2
\right)^{1/2}.
\end{align*}
By Lemma~\ref{l:hilbert-barycenter-convergence}, the first factor
is \(o(1)\). For the second factor, using
\((a+b)^2\leq 2a^2+2b^2\), we have
\begin{align*}
\mathbb E
\left(
\|X_N\bm_N\|_{\mathcal H_N}
+\|\bm_N^\sph\|_{\mathcal H_N}
\right)^2
&\leq
2\mathbb E\left[X_N^2L_N(t)\right]
+2\mathbb E\left[L_N^\sph(t)\right]
=O(1),
\end{align*}
where the last bound follows from
\eqref{e:annealed-exponential-localization}.
Consequently,
\[
\mathbb E\left[
\left|X_N^2L_N(t)-L_N^\sph(t)\right|
\right]
\rightarrow 0\qquad\text{as }N\to\infty\,.
\]
In particular, by Markov's inequality,
\[
X_N^2L_N(t)-L_N^\sph(t)
\rightarrow0
\]
in probability.
By the symmetry \(L_N(-t)=L_N(t)\) and
\(L_N^\sph(-t)=L_N^\sph(t)\), the same conclusion holds for every
\(t\in(-t_0,t_0)\).
Finally, using
\[
L_N(t)-L_N^\sph(t)
=
\bigl(X_N^2L_N(t)-L_N^\sph(t)\bigr)
-(X_N^2-1)L_N(t),
\]
combined with the facts that \(X_N\to1\) in probability and \(L_N(t)\) is bounded in probability,
we conclude that
\[
L_N(t)-L_N^\sph(t)\rightarrow0
\]
in probability.
\end{proof}

\section{Identification of the overlap distribution}
\label{s:overlap}

\subsection{Convergence of the augmented edge data}
\label{ss:augmented-edge-convergence}

The goal of this subsection is to prove the following proposition:

\begin{ppn}\label{p:augmented-edge-convergence}
  Recall the quantities $\chi_{N,k}$, $d_{N,k}$, $\Delta_N$, and $a_{N,k}$ from \S\ref{ss:intro-spherical}.
  On the probability-one event that the GOE spectrum is simple, set
  \[
    \Xi_N=\sum_{k=2}^N\fr1{d_{N,k}}-N^{1/3},
    \qquad
    b_N=(0,d_{N,2}^{-1},\ldots,d_{N,N}^{-1},0,\ldots)\in\ell^2.
  \]
  On the null event of a repeated eigenvalue, define these objects arbitrarily.
  On the Airy side, put
  \[
    \Xi=\Xi(\chi),
    \qquad
    b=(0,d_2^{-1},d_3^{-1},\ldots).
  \]
  Then, as $N\to\infty$,
  \beq\label{e:augmented-edge-convergence}
    (\Xi_N,b_N)\Rightarrow (\Xi(\chi),b)
    \qquad\text{in }\bbR\times\ell^2.
  \eeq
  Consequently, with $a_N=(a_{N,1},\ldots,a_{N,N},0,\ldots)$,
  \beq\label{e:coefficient-l2-convergence}
    a_N\Rightarrow a(\chi)
    \qquad\text{in }\ell^2.
  \eeq
\end{ppn}

We will first prove \eqref{e:augmented-edge-convergence} and then deduce \eqref{e:coefficient-l2-convergence} via a continuous mapping argument. Towards establishing \eqref{e:augmented-edge-convergence}, we will need a finite-truncation approximation of $\Xi_N$ as well as inverse-gap tail estimates, which are detailed in Lemmas~\ref{l:finite-truncation} and~\ref{l:inverse-gap-tail} below separately.

We start by introducing the finite-truncation approximation. Recall that $t_k=\left(\frac{3\pi k}{2}\right)^{2/3}$, $k=1,2,\ldots$.
For $K\ge2$, define
\[
  c_K=\fr1\pi\int_0^{t_K}x^{-1/2}\,\de x,
  \qquad
  \Xi_{N,K}=\sum_{k=2}^K\fr1{d_{N,k}}-c_K,
  \qquad
  \Xi_K=\sum_{k=2}^K\fr1{d_k}-c_K.
\]
Let $\rho_{\mathrm{sc}}$ denote the semicircle law
\[
  \rho_{\mathrm{sc}}(x)
  =
  \fr1{2\pi}\sqrt{4-x^2}\,\bbone_{\{|x|\le2\}}.
\]
Heuristically, for large $K$, one expects that with high probability $d_{N,K}\approx t_K$, and thus
\[
\sum_{k=K+1}^N\frac{1}{d_{N,k}}\approx \sum_{d_{N,k}>t_K}\frac{1}{d_{N,k}}\approx N^{1/3}\int_{-2}^{2-N^{-2/3}t_K}\frac{\rho_{\mathrm{sc}}(x)}{2-x}\de x\approx N^{1/3}-c_K\,,
\]
where the last approximation uses the facts that $\int_{-2}^2\frac{\rho_{\mathrm{sc}}(x)}{2-x}\de x=1$ and $\frac{\rho_{\mathrm{sc}}(x)}{2-x}\sim \frac{1}{\pi}(2-x)^{-1/2}$ as $x\uparrow 2$. Therefore, $\Xi_{N,K}$ should be a good approximation of $\Xi_N$ for large $K$. The next lemma makes this precise.

\begin{lem}[Finite-truncation approximation]\label{l:finite-truncation}
  For every $\eps>0$,
  \beq\label{e:XiN-truncation}
    \lim_{K\to\infty}\limsup_{N\to\infty}
    \bbP\bigl(|\Xi_N-\Xi_{N,K}|>\eps\bigr)=0.
  \eeq
\end{lem}

\begin{proof}
  Define the upper-edge locations $\hat\gamma_{k,N}$ by
  \[
    \int_{\hat\gamma_{k,N}}^2\rho_{\mathrm{sc}}(x)\,\de x
    =
    \fr{k}{N},
    \qquad
    0\le k\le N,
  \]
  where $\hat\gamma_{0,N}=2$.
  With our decreasing eigenvalue ordering, the upper-edge location obtained from the convention in \cite[Theorem~3.3]{landon2022fluctuations} is
  $\overline\gamma_{k,N}=\hat\gamma_{k-1,N}$.
 
  Denote $M_N=\lfloor N^{1/20}\rfloor$.
  For $2\le K<M_N$, define
  \[
    \begin{split}
      B_{N,K}
      &=
      \sum_{k=K+1}^{M_N}\fr1{d_{N,k}}
      -\fr1\pi\int_{t_K}^{t_{M_N}}x^{-1/2}\,\de x,\\
      C_N
      &=
      \sum_{k=M_N+1}^N\fr1{d_{N,k}}
      -N^{1/3}\int_{-2}^{\hat\gamma_{M_N,N}}
        \fr{\rho_{\mathrm{sc}}(x)}{2-x}\,\de x,\\
      D_N
      &=
      \fr1\pi\int_0^{t_{M_N}}x^{-1/2}\,\de x
      -N^{1/3}\int_{\hat\gamma_{M_N,N}}^2
        \fr{\rho_{\mathrm{sc}}(x)}{2-x}\,\de x.
    \end{split}
  \]
  Since
  $
    \int_{-2}^2\fr{\rho_{\mathrm{sc}}(x)}{2-x}\,\de x=1,
  $
  these definitions give the exact decomposition
  \beq\label{e:XiN-tail-decomposition}
    \Xi_N-\Xi_{N,K}=B_{N,K}+C_N+D_N.
  \eeq
  We will show that each of the three terms in the right-hand side of \eqref{e:XiN-tail-decomposition} converges to $0$ while taking $N\to \infty$ and subsequently $K\to\infty$.

  For the term $B_{N,K}$, we may write
  \[
    B_{N,K}
    =
    \sum_{k=K+1}^{M_N}
    \lt[
      \fr1{N^{2/3}(\lambda_1-\lambda_k)}
      -
      \fr1\pi\int_{t_{k-1}}^{t_k}x^{-1/2}\,\de x
    \rt].
  \]
  Thus \cite[Equation (6.40)]{landon2022fluctuations}, with its edge index equal to $K+1$, implies that for every $r>0$ there is $K_r$ such that
  \beq\label{e:edge-block-control}
    \limsup_{N\to\infty}\bbP(|B_{N,K}|>r)\le r,
    \qquad
    K\ge K_r.
  \eeq

  We next handle the term $C_N$ by applying \cite[Equation (6.41)]{landon2022fluctuations} with $N^{\delta_0}$ replaced by $M_N+1$.
  For every $D>0$, that estimate gives $c>0$ and an event $\cR_N^{\mathrm{res}}$ with $\bbP(\cR_N^{\mathrm{res}})\ge 1-N^{-D}$ on which
  \[
    \lt|
      \fr1N\sum_{k=M_N+1}^N\fr1{\lambda_1-\lambda_k}
      -
      \int_{-2}^{\overline\gamma_{M_N+1,N}}
        \fr{\rho_{\mathrm{sc}}(x)}{2-x}\,\de x
    \rt|
    \le N^{-1/3-c/20}.
  \]
  Since $d_{N,k}=N^{2/3}(\lambda_1-\lambda_k)$ and $\overline\gamma_{M_N+1,N}=\hat\gamma_{M_N,N}$, this means precisely that
  \begin{equation}\label{e:resolvent-event-probability}
  |C_N|\le N^{-c/20}\quad \text{on }\mathcal R_N^{\mathrm{res}}. 
  \end{equation}
 
It remains to estimate $D_N$, which is a deterministic quantity.
  For $M\le N^{2/5}$, set
  \[
    s_{M,N}=2-\hat\gamma_{M,N},
    \qquad
    T_{M,N}=N^{2/3}s_{M,N}.
  \]
  Taylor expansion of the semicircle density at the upper edge gives
  \[
    \fr{M}{N}
    =
    \fr{2}{3\pi}s_{M,N}^{3/2}\lt(1+O(s_{M,N})\rt),
    \qquad
    \fr{\rho_{\mathrm{sc}}(2-u)}{u}
    =
    \fr1\pi u^{-1/2}\lt(1+O(u)\rt).
  \]
  The first expansion and the definition $t_M=(3\pi M/2)^{2/3}$ imply
  \beq\label{e:classical-edge-scaled-taylor}
    T_{M,N}
    =
    t_M+O(t_M^2N^{-2/3}).
  \eeq
  Using the second expansion and the change of variables $u=2-x$, we obtain
  \beq\label{e:semicircle-resolvent-taylor}
    \begin{split}
      N^{1/3}\int_{\hat\gamma_{M,N}}^2
        \fr{\rho_{\mathrm{sc}}(x)}{2-x}\,\de x
      &=
      N^{1/3}\int_0^{s_{M,N}}
        \fr{\rho_{\mathrm{sc}}(2-u)}{u}\,\de u=
      \fr2\pi T_{M,N}^{1/2}
      +O(MN^{-2/3}).
    \end{split}
  \eeq
  Equation \eqref{e:classical-edge-scaled-taylor} gives
  \[
    |T_{M,N}^{1/2}-t_M^{1/2}|
    =
    O(t_M^{3/2}N^{-2/3})
    =
    O(MN^{-2/3}).
  \]
  Since $\pi^{-1}\int_0^{t_M}x^{-1/2}\,\de x=2\pi^{-1}t_M^{1/2}$, \eqref{e:semicircle-resolvent-taylor} yields
  \beq\label{e:deterministic-resolvent-control}
    D_N=O(M_NN^{-2/3})=o(1).
  \eeq

  We now combine the three estimates in \eqref{e:XiN-tail-decomposition}.
  Given $\eps,\eta>0$, choose $r<\min\{\eps/3,\eta\}$ and then choose $K\ge K_r$.
  Equations \eqref{e:edge-block-control}, \eqref{e:resolvent-event-probability}, and \eqref{e:deterministic-resolvent-control} then give
  \[
    \limsup_{N\to\infty}
    \bbP(|\Xi_N-\Xi_{N,K}|>\eps)
    \le r<\eta.
  \]
  Letting $K\to\infty$ proves \eqref{e:XiN-truncation}.
\end{proof}

The next lemma controls the tail of the random sequence $b_N=(0,d_{N,2}^{-1},\dots,d_{N,N}^{-1},0,\dots)\in \ell^2$.

\begin{lem}[Inverse-gap tails]\label{l:inverse-gap-tail}
  For every $\eps>0$,
  \beq\label{e:inverse-gap-l2-tail}
    \lim_{K\to\infty}\limsup_{N\to\infty}
    \bbP\lt(\sum_{K<k\le N}d_{N,k}^{-2}>\eps\rt)=0.
  \eeq
\end{lem}

\begin{proof}
  Fix $A,D>0$ and $0<\zeta<2/5$, and choose
  \[
    \alpha=\lt(\fr{3\pi}{2}\rt)^{2/3}-\fr1{10}>0.
  \]
  By the reflected form of \cite[Equation (6.33)]{landon2022fluctuations}, there is a fixed $K_1$ such that, for every integer $K_1\le q\le N^{2/5}$,
  \begin{equation}\label{e:tail-Gq}
    \cG_{N,q}
    =
    \bigcap_{q\le k\le N^{2/5}}
    \{\chi_{N,k}\ge\alpha q^{2/3}\},
    \qquad
    \bbP(\cG_{N,q}^c)\le q^{-1/2}.
  \end{equation}
  Moreover, by the classical GOE eigenvalue rigidity \cite{ErdosYauYin2012Rigidity} (see also \cite[Theorem~3.3]{landon2022fluctuations}), for every $\zeta,D>0$, there is an event $\cR_{N,\zeta,D}$ with $\bbP(\cR_{N,\zeta,D}^c)\le N^{-D}$ on which
  \beq\label{e:upper-edge-rigidity}
    |N^{2/3}(\lambda_k-\hat\gamma_{k,N})|
    \le C N^{\zeta}k^{-1/3},
    \qquad
    2\le k\le N/2.
  \eeq
  We fix $\zeta=0.1$ and $D=10$. For $K\ge K_1$, set $q_\ell=\lceil2^\ell K\rceil$ for $\ell=0,1,\ldots$, and define
  \[
    \cE_{N,K,A}
    =
    \{|\chi_{N,1}|\le A\}
    \cap
    \bigcap_{q_\ell\le N^{2/5}}\cG_{N,q_\ell}
    \cap\cR_{N,\zeta,D}.
  \]
  Combining the union bound with \eqref{e:tail-Gq} gives
  \beq\label{e:inverse-gap-event-probability}
    \bbP(\cE_{N,K,A}^c)
    \le
    \bbP(|\chi_{N,1}|>A)+CK^{-1/2}+N^{-D}.
  \eeq
  
  If $K$ is larger than a constant multiple of $A^{3/2}$, then on $\cE_{N,K,A}$, for every $K<k\le N^{2/5}$, we may choose $\ell$ with $q_\ell\le k<q_{\ell+1}$ and obtain
  \[
    d_{N,k}
    \ge
    \alpha q_\ell^{2/3}-A
    \ge
    c q_\ell^{2/3}
    \ge
    c'k^{2/3}.
  \]
  Consequently,
  \[
    \sum_{K<k\le N^{2/5}}d_{N,k}^{-2}
    \le
    C\sum_{k>K}k^{-4/3}
    \le
    CK^{-1/3}.
  \]
  For $N^{2/5}<k\le N/2$, the standard upper-edge estimate gives
  $N^{2/3}(2-\hat\gamma_{k,N})\ge ck^{2/3}$.
  Together with \eqref{e:upper-edge-rigidity} and $|\chi_{N,1}|\le A$, this implies that on $\cE_{N,K,A}$,
  \[
    d_{N,k}
    \ge
    ck^{2/3}-A-CN^\zeta k^{-1/3}
    \ge
    c'k^{2/3},
  \]
  because $N^\zeta/k\le N^{\zeta-2/5}=o(1)$.
  Monotonicity and the same estimate at $k=\lfloor N/2\rfloor$ give $d_{N,k}\ge cN^{2/3}$ for $k>N/2$.
  Therefore
  \[
    \sum_{N^{2/5}<k\le N}d_{N,k}^{-2}
    \le
    C\sum_{N^{2/5}<k\le N/2}k^{-4/3}
    +CN\cdot N^{-4/3}
    =
    O(N^{-2/15}+N^{-1/3}).
  \]
  First choose $A$ so that the first probability in \eqref{e:inverse-gap-event-probability} is small, then choose $K\gg A^{3/2}$, and finally take the $N$-limit superior.
  This proves \eqref{e:inverse-gap-l2-tail}.
\end{proof}

We are now ready to prove Proposition~\ref{p:augmented-edge-convergence}.

\begin{proof}[Proof of Proposition~\ref{p:augmented-edge-convergence}]
  We first prove joint convergence of $\Xi_N$ with every fixed collection of gaps.
  Fix $m\ge2$ and $K\ge m$, and put
  \[
    \begin{aligned}
      U_N&=(\Xi_N,d_{N,2},\ldots,d_{N,m}),&
      U_{N,K}&=(\Xi_{N,K},d_{N,2},\ldots,d_{N,m}),\\
      U&=(\Xi,d_2,\ldots,d_m),&
      U_K&=(\Xi_K,d_2,\ldots,d_m).
    \end{aligned}
  \]
  The standard joint convergence of the first $K$ GOE edge eigenvalues underlying \eqref{e:airy-convergence} and the continuous mapping theorem give
  \beq\label{e:fixed-truncation-gap-convergence}
    U_{N,K}\Rightarrow U_K.
  \eeq
  Let $H:\bbR^m\to\bbR$ be bounded by one and $1$-Lipschitz.
  For every $\eps>0$,
  \[
    \big|\bbE H(U_N)-\bbE H(U_{N,K})\big|
    \le
    \eps+2\bbP(|\Xi_N-\Xi_{N,K}|>\eps).
  \]
  The same estimate on the Airy probability space gives
  \[
    \big|\bbE H(U_K)-\bbE H(U)\big|
    \le
    \eps+2\bbP(|\Xi_K-\Xi|>\eps).
  \]
  For fixed $K$, the middle comparison tends to zero by \eqref{e:fixed-truncation-gap-convergence}.
  Next let $K\to\infty$, using Lemma~\ref{l:finite-truncation} for the finite-dimensional error and the almost-sure convergence $\Xi_K\to\Xi$ from Proposition~\ref{p:Xi} for the limiting error.
  Finally let $\eps\downarrow0$.
  Since bounded Lipschitz functions determine weak convergence on $\bbR^m$, we have
  \beq\label{e:finite-gap-joint-convergence}
    (\Xi_N,d_{N,2},\ldots,d_{N,m})
    \Rightarrow
    (\Xi,d_2,\ldots,d_m).
  \eeq

  Recall that $b_N=(0,d_{N,2}^{-1},\dots,d_{N,N}^{-1},0,\dots)\in \ell^2$ and $b=(0,d_2^{-1},d_3^{-1},\dots)\in \ell^2$ $\chi$-almost surely. Let $P_m$ retain the coordinates of $\ell^2$ indexed by $1,\ldots,m$.
  Simplicity of the limiting point process (i.e. $d_k\neq 0$ for all $k\ge2$ $\chi$-almost surely) and \eqref{e:finite-gap-joint-convergence} give for every fixed $m$,
  \[
    (\Xi_N,P_mb_N)\Rightarrow(\Xi,P_mb).
  \]
  If $H$ is bounded by one and $1$-Lipschitz on $\bbR\times\ell^2$, then for every $\eta>0$,
  \[
    \begin{split}
      &\ \qquad\qquad\qquad\qquad\qquad\qquad\big|\bbE H(\Xi_N,b_N)-\bbE H(\Xi,b)\big|
      \\\le&\ 
      \eta+2\bbP\bigl(\|(I-P_m)b_N\|_2>\eta\bigr)+
      \big|\bbE H(\Xi_N,P_mb_N)-\bbE H(\Xi,P_mb)\big|+
      \eta+2\bbP\bigl(\|(I-P_m)b\|_2>\eta\bigr).
    \end{split}
  \]
  Take first $N\to\infty$, then $m\to\infty$, using \eqref{e:inverse-gap-l2-tail} from Lemma~\ref{l:inverse-gap-tail} with threshold $\eta^2$ and the almost-sure $\ell^2$ tail convergence of $b$, and finally let $\eta\downarrow0$.
  This proves \eqref{e:augmented-edge-convergence}.

It remains to pass from the inverse gaps to the saddle coefficients by realizing $\Delta_N$, and hence $a_N$, as continuous functions of $(\Xi_N,b_N)$. Let
  \[
    \ell^2_{+,0}
    =
    \{v\in\ell^2:v_1=0,\ v_k\ge0\text{ for }k\ge2\},
    \qquad
    \cX=\bbR\times\ell^2_{+,0}.
  \]
  The cone $\ell^2_{+,0}$ is closed in $\ell^2$, so $\cX$ is Polish.
  Since both input pairs are $\cX$-valued almost surely, \eqref{e:augmented-edge-convergence} also holds when they are regarded as random elements of $\cX$.
  For $(\xi,v)\in\cX$ and $\delta>0$, define
  \[
    \Psi_{\xi,v}(\delta)
    =
    \xi+\fr1\delta
    -\sum_{k\ge2}\fr{\delta v_k^2}{1+\delta v_k}\,.
  \]
  The map $(\xi,v,\delta)\mapsto\Psi_{\xi,v}(\delta)$ is jointly continuous, locally uniformly for $\delta\in(0,\infty)$.
  Indeed, for $\delta\in[\eta,L]$ and $x,y\ge0$,
  \[
    \lt|
      \fr{\delta x^2}{1+\delta x}
      -
      \fr{\delta y^2}{1+\delta y}
    \rt|
    \le
    2L(x+y)|x-y|,
  \]
  so Cauchy--Schwarz controls the series.
  Similarly, for $\delta,\delta'\in[\eta,L]$,
  \[
    \lt|
      \fr{\delta x^2}{1+\delta x}
      -
      \fr{\delta'x^2}{1+\delta'x}
    \rt|
    \le
    |\delta-\delta'|x^2.
  \]
  Termwise differentiation is justified by $\sum_kv_k^2<\infty$, and
  \[
    \Psi_{\xi,v}'(\delta)
    =
    -\fr1{\delta^2}
    -\sum_{k\ge2}\fr{v_k^2}{(1+\delta v_k)^2}
    <0.
  \]

  Define
  \[
    \cD
    =
    \bigcup_{\substack{r,s\in\bbQ\\0<r<s}}
    \{(\xi,v)\in\cX:\Psi_{\xi,v}(r)>0>\Psi_{\xi,v}(s)\}.
  \]
  The set $\cD$ is open, and hence Borel, in $\cX$.
  For $(\xi,v)\in\cD$, strict monotonicity and the intermediate value theorem give a unique zero, denoted by $\operatorname{Root}(\xi,v)$.
  Conversely, if $\Psi_{\xi,v}$ has a zero $\delta$, strict monotonicity and density of the rationals give $0<r<\delta<s$ with $\Psi_{\xi,v}(r)>0>\Psi_{\xi,v}(s)$, so $(\xi,v)\in\cD$.
  The root map is continuous on $\cD$.
  Indeed, if $(\xi_n,v_n)\to(\xi,v)$ in $\cD$ and $\delta=\operatorname{Root}(\xi,v)$, then for every $0<\eta<\delta<L$, local uniform convergence gives $\Psi_{\xi_n,v_n}(\eta)>0>\Psi_{\xi_n,v_n}(L)$ for all sufficiently large $n$.
  Thus $\operatorname{Root}(\xi_n,v_n)\in(\eta,L)$, and letting $\eta\uparrow\delta$ and $L\downarrow\delta$ proves convergence of the roots.

  For the limiting pair, $\Psi_{\Xi,b}(\delta)=\Psi(\delta;\chi)$, so Proposition~\ref{p:limit-distribution}\ref{i:limit-distribution-Psi} implies $(\Xi,b)\in\cD$ almost surely and $\operatorname{Root}(\Xi,b)=\Delta(\chi)$.
  For the finite pair,
  \[
    \Psi_{\Xi_N,b_N}(\delta)
    =
    \sum_{k=1}^N\fr1{\delta+d_{N,k}}-N^{1/3},
  \]
  so \eqref{e:gamma-def} implies $(\Xi_N,b_N)\in\cD$ almost surely and $\operatorname{Root}(\Xi_N,b_N)=\Delta_N$.

  Define $\mathsf A:\cD\to\ell^2$ by
  \[
    \mathsf A(\xi,v)_1
    =
    \operatorname{Root}(\xi,v)^{-1},
    \qquad
    \mathsf A(\xi,v)_k
    =
    \fr{v_k}{1+\operatorname{Root}(\xi,v)v_k},
    \quad k\ge2.
  \]
  This map is continuous.
  To see this, if $(\xi_n,v_n)\to(\xi,v)$ in $\cD$ and the corresponding roots are $\delta_n,\delta$, then
  \[
    \lt\|
      \fr{v_n}{1+\delta_nv_n}
      -
      \fr v{1+\delta v}
    \rt\|_2
    \le
    \|v_n-v\|_2
    +
    |\delta_n-\delta|
    \lt(\sum_{k\ge2}v_k^4\rt)^{1/2}.
  \]
  The last sum is finite because $\sum_kv_k^4\le\|v\|_2^4$, and the first coordinate also converges.
  Set $a^\star=(t_k^{-1})_{k\ge1}\in\cA$, and extend $\mathsf A$ by $a^\star$ on $\cX\setminus\cD$.
  The extension is Borel and is continuous at every point of $\cD$, since $\cD$ is open.
  Since the limiting pair lies in $\cD$ almost surely, the continuous mapping theorem applied to \eqref{e:augmented-edge-convergence} gives
  \[
    \mathsf A(\Xi_N,b_N)
    \Rightarrow
    \mathsf A(\Xi,b).
  \]
  The two sides are $a_N$ and $a(\chi)$ almost surely, respectively, which proves \eqref{e:coefficient-l2-convergence}.
  Because $\chi\mapsto(\Xi,b)$ has a Borel version and the extended root/coefficient map $\mathsf A$ is Borel, this construction first provides an $\ell^2$-valued Borel version of $\chi\mapsto a(\chi)$ which belongs to $\cA$ almost surely.
  Since $\cA$ is Borel, replacing this version by $a^\star$ wherever it lies outside $\cA$ provides an everywhere $\cA$-valued Borel version.
\end{proof}

\subsection{Finite spherical pins}
\label{ss:annealed-limit-square-integrability}

The goal of this subsection is to prove Proposition~\ref{p:limit-distribution-square-integrable-random}.
We first identify the finite Gaussian pin associated with $a_N$ with the spherical Gibbs measure and record the measurability of its overlap law and second moment.
The coefficient convergence from Proposition~\ref{p:augmented-edge-convergence}, continuity of pinned laws, and the spherical exponential moment bound then yield the required annealed estimate.

We begin with some measurability discussion for the finite pins, which ensures that for any $\ell^2$-valued random variable $\mathbf{c}$, the extensions $\fP_{\mathbf{c}}^{\mathrm{fin}}$ and $V_{\mathbf{c}}^{\mathrm{fin}}$ defined below are well-defined random variables in $\mathcal P_2(\bbR)$ and $\bbR$.
For a finite positive vector $c=(c_1,\ldots,c_M)$ with $M\ge3$, let $\fP_c$ be the law of
\[
  \sum_{k=1}^M c_k g_k^{(1)}g_k^{(2)}
\]
when the two sequences are independent samples from the finite pin \eqref{e:finite-coarea-pin}.
Recall the definition of $V_c$ following \eqref{e:finite-coarea-pin}.
For $M\ge3$, set
\[
  \cF_M
  =
  \{c\in\ell^2:c_k>0\text{ for }1\le k\le M,\ c_k=0\text{ for }k>M\},
  \qquad
  \cF_{\mathrm{fin}}=\bigcup_{M\ge3}\cF_M.
\]

\begin{lem}[Measurability of finite pins]
  \label{l:finite-pin-measurability}
  The set $\cF_{\mathrm{fin}}$ is Borel in $\ell^2$.
  On each stratum $\cF_M$, the maps
  \[
    c\mapsto\fP_c\in(\cP_2(\bbR),\bbW_2),
    \qquad
    c\mapsto V_c
  \]
  are continuous, and they are Borel on $\cF_{\mathrm{fin}}$.
  Consequently, the extensions defined by
  \[
    (\fP_c^{\mathrm{fin}},V_c^{\mathrm{fin}})
    =
    \begin{cases}
      (\fP_c,V_c),&c\in\cF_{\mathrm{fin}},\\
      (\delta_0,0),&c\notin\cF_{\mathrm{fin}},
    \end{cases}
  \]
  are Borel on $\ell^2$.
\end{lem}

\begin{proof}
  Each stratum is Borel because
  \[
    \cF_M
    =
    \bigcap_{k\le M}\{c\in\ell^2:c_k>0\}
    \cap
    \bigcap_{k>M}\{c\in\ell^2:c_k=0\}.
  \]
  The strata are disjoint, so their countable union is Borel.

  Fix $M\ge3$, let $c\in\cF_M$, and put $\alpha_c=\sum_{k=1}^M c_k$.
  Lemma~\ref{l:diagonal-coarea}, followed by the scaling $y=\sqrt{\alpha_c}\,u$, gives the following law on $\bbS^{M-1}$:
  \beq\label{e:finite-pin-angular-law}
    \rho_c(\de u)
    =
    \fr{
      \exp\{-\fr{\alpha_c}{2}\sum_{k=1}^M u_k^2/c_k\}
    }{
      \displaystyle
      \int_{\bbS^{M-1}}
      \exp\{-\fr{\alpha_c}{2}\sum_{k=1}^M v_k^2/c_k\}
      \,\de\omega_1(v)
    }
    \,\de\omega_1(u).
  \eeq
  If $u,v$ are independent with law $\rho_c$, then
  \beq\label{e:finite-pin-angular-overlap}
    \fP_c
    =
    \Law_{\rho_c^{\otimes2}}(\alpha_c\,u\cdot v).
  \eeq
  Coordinatewise sign symmetry and replica independence give
  \[
    \int q^2\,\fP_c(\de q)
    =
    \sum_{k=1}^Mc_k^2m_k(c)^2
    =
    V_c.
  \]

  Suppose $c^{(n)}\to c$ in $\cF_M$.
  The exponents and normalizing constants in \eqref{e:finite-pin-angular-law} converge uniformly on $\bbS^{M-1}$.
  Thus the densities of $\rho_{c^{(n)}}$ converge uniformly to the density of $\rho_c$.
  Moreover, $\alpha_{c^{(n)}}u\cdot v\to\alpha_cu\cdot v$ uniformly on $(\bbS^{M-1})^2$, and these functions are uniformly bounded.
  It follows that the laws in \eqref{e:finite-pin-angular-overlap} converge weakly and that their second moments converge.
  The weak-plus-second-moment characterization of Wasserstein convergence (see, e.g., \cite[Theorem 6.9]{Villani2009}) gives
  \[
    \bbW_2(\fP_{c^{(n)}},\fP_c)\rightarrow0.
  \]
  Since
  \[
    V_c=\int q^2\,\fP_c(\de q),
  \]
  the map $c\mapsto V_c$ is also continuous on $\cF_M$.
  A map on a countable union of Borel strata is Borel when its restriction to each stratum is Borel.
  The stated extensions are therefore Borel on $\ell^2$.
\end{proof}

The next lemma connects the quenched spherical Gibbs measure with finite Gaussian pins.

\begin{lem}
  \label{l:finite-spherical-pin}
  For $N\ge3$, the vector $a_N=(a_{N,1},\ldots,a_{N,N},0,\ldots)$ belongs to $\cF_N$ almost surely, and conditionally on $\bW$,
  \beq\label{e:finite-pin-overlap-law}
    \fP_{a_N}
    =
    \la\delta(N^{1/3}R_{1,2})\ra^\sph.
  \eeq
  Moreover,
  \beq\label{e:finite-pin-overlap-second-moment}
    V_{a_N}
    =
    N^{2/3}\la R_{1,2}^2\ra^\sph.
  \eeq
\end{lem}

\begin{proof}
  The saddle equation and the definition of $a_N$ give
  \beq\label{e:finite-aN-sum}
    \sum_{k=1}^Na_{N,k}=N^{1/3}.
  \eeq
  Apply Lemma~\ref{l:diagonal-coarea} with $c=(a_{N,1},\ldots,a_{N,N})$, and set
  \[
    y_k=a_{N,k}^{1/2}g_k,
    \qquad
    \xi_k=N^{1/3}y_k.
  \]
  By \eqref{e:finite-aN-sum}, $y$ lies on the sphere of squared radius $N^{1/3}$ and $\bxi$ lies on $S_N$.
  Since
  \[
    a_{N,k}^{-1}
    =
    \Delta_N+d_{N,k}
    =
    N^{2/3}(\gamma-\lambda_k),
  \]
  the density furnished by Lemma~\ref{l:diagonal-coarea}, after the second scaling, is proportional on $S_N$ to
  \[
    \exp\lt\{-\fr12\sum_{k=1}^N(\gamma-\lambda_k)\xi_k^2\rt\}
    =
    e^{-\gamma N/2}
    \exp\lt\{\fr12\sum_{k=1}^N\lambda_k\xi_k^2\rt\}.
  \]
  The first factor is constant on $S_N$.
  Thus the pushforward of the finite pin is exactly $\mu^\sph_{N,\beta=1}$ in an eigenbasis of $\bW$.
  Orthogonal invariance of surface measure and of the overlap gives the same identification in the original coordinates.
  Under this pushforward coupling, for two independent replicas,
  \[
    N^{1/3}R_{1,2}
    =
    \sum_{k=1}^Na_{N,k}g_k^{(1)}g_k^{(2)},
  \]
  which proves \eqref{e:finite-pin-overlap-law}.
  Coordinatewise sign symmetry of the finite pin gives $\bbE_{\nu_{a_N}}g_jg_k=0$ for $j\ne k$.
  Squaring the last display and using replica independence identifies its second moment with $V_{a_N}$ and proves \eqref{e:finite-pin-overlap-second-moment}.
\end{proof}

\begin{proof}[Proof of Proposition~\ref{p:limit-distribution-square-integrable-random}]
  By Proposition~\ref{p:augmented-edge-convergence}, fix an everywhere $\cA$-valued Borel version of $\chi\mapsto a(\chi)$.
  Lemma~\ref{l:parameter-measurability} gives Borel maps $a\mapsto\fP_a$ and $a\mapsto V_a$ on $\cA$.

  By Proposition~\ref{p:augmented-edge-convergence} and Skorokhod's representation theorem, there are copies $\widetilde a_N$ and $\widetilde a$ of $a_N$ and $a(\chi)$ on a common probability space such that
  $\widetilde a_N\rightarrow\widetilde a$ in $\ell^2$
  almost surely.
  On a common probability-one event, every $\widetilde a_N$ belongs to $\cF_N$, $\widetilde a\in\cA$, and the displayed convergence holds.
  Lemma~\ref{l:finite-pin-measurability} makes $V_{\widetilde a_N}^{\mathrm{fin}}$ a nonnegative random variable.
  Lemma~\ref{l:parameter-measurability} makes $V_{\widetilde a}$ measurable on this event, and we set it equal to zero on its null complement.
  Lemma~\ref{l:pinned-law-continuity} below applies pathwise and gives
  $
    V_{\widetilde a_N}^{\mathrm{fin}}
    \rightarrow
    V_{\widetilde a}.
  $
  Fatou's lemma and preservation of the marginal laws on the Skorokhod coupling yield
  \beq\label{e:annealed-V-liminf}
    \bbE_\chi V_{a(\chi)}
    =
    \widetilde{\bbE}V_{\widetilde a}
    \le
    \liminf_{N\to\infty}
    \widetilde{\bbE}V_{\widetilde a_N}^{\mathrm{fin}}
    =
    \liminf_{N\to\infty}\bbE V_{a_N}.
  \eeq

  It remains to prove the right-hand side of \eqref{e:annealed-V-liminf} is finite. Let $c>0$ be the constant in Proposition~\ref{p:exponential-moment-bound}.
  Since
  \[
    x^2\le\fr4{c^2e^2}e^{c|x|},
    \qquad\forall x\in\bbR,
  \]
  that proposition and Lemma~\ref{l:finite-spherical-pin} imply, for every $N\ge3$,
  \[
    \bbE V_{a_N}
    =
    \bbE\la(N^{1/3}R_{1,2})^2\ra^\sph
    \le
    \fr8{c^2e^2}.
  \]
  Equation \eqref{e:annealed-V-liminf} proves $\bbE_\chi V_{a(\chi)}<\infty$.

  Finally,
  \[
    \bbW_2^2(\fP_{a(\chi)},\delta_0)
    =
    \int q^2\,\fP_{a(\chi)}(\de q)
    =
    V_{a(\chi)}.
  \]
  The map $\chi\mapsto\fP_{a(\chi)}$ is Borel by Proposition~\ref{p:augmented-edge-convergence} and Lemma~\ref{l:parameter-measurability}, while the preceding expectation bound is precisely the required outer second-moment condition.
  Hence $\Law(\fP_{a(\chi)})\in\cP_2(\cP_2(\bbR))$.
\end{proof}

\subsection{Continuity of pinned laws and the spherical overlap limit}
\label{ss:spherical-overlap-limit}

The goal of this subsection is to prove Theorem~\ref{t:main}\ref{i:main-sphere}.
We first prove that convergence in $\ell^2$ implies that finite pins converge to the infinite pin in $\bbW_2$ (Lemma~\ref{l:pinned-law-continuity}), by combining convergence of their finite-coordinate marginals with uniform tail control.
The finite spherical-pin identity and Proposition~\ref{p:augmented-edge-convergence} then give weak convergence of the random overlap laws, which relative compactness upgrades to convergence in $\cW_2$.

To handle the finite-coordinate marginals, we record a finite-pin analogue of the density formula \eqref{e:canonical-pin-marginal-s}. 

\begin{lem}[Finite-pin block marginals]
  \label{l:finite-pin-block-marginals}
  Let $c=(c_1,\ldots,c_M)\in(0,\infty)^M$, and let $1\le K\le M-3$.
  Write
  \[
    S_{c,K}(z)=\sum_{k=1}^Kc_k(z_k^2-1),
  \]
  and let $p_{c,>K}$ be the density of
  \[
    T_{c,>K}=\sum_{k=K+1}^Mc_k(g_k^2-1).
  \]
  Then $p_{c,>K}$ is bounded and continuous, and the first-$K$ marginal of $\nu_c$ has density
  \beq\label{e:finite-pin-block-density}
    h_{c,K}(z)
    =
    \fr{\varphi_K(z)}{p_c(0)}
    p_{c,>K}\lt(-S_{c,K}(z)\rt),
  \eeq
  where, for every $r\ge1$,
  \[
    \varphi_r(z)=\prod_{k=1}^r\fr{e^{-z_k^2/2}}{\sqrt{2\pi}}.
  \]
\end{lem}

\begin{proof}
  Put $\alpha_c=\sum_{k=1}^Mc_k$.
  Choose an even nonnegative function $\eta\in C_c^\infty(\bbR)$ with $\int\eta=1$, and set $\eta_\eps(t)=\eps^{-1}\eta(t/\eps)$.
  For a bounded continuous function $\vartheta:\bbR^K\to\bbR$, define, for $t>-\alpha_c$,
  \[
    A_\vartheta(t)
    =
    \int_{F_c^{-1}(t)}
    \vartheta(x_{\le K})
    \fr{\varphi_M(x)}{\|\nabla F_c(x)\|}
    \,\de\cH^{M-1}(x).
  \]
  The function $A_\vartheta$ is continuous in a neighborhood of zero.
  Indeed, if $r_t=(1+t/\alpha_c)^{1/2}$, then the map $x\mapsto r_tx$ carries $F_c^{-1}(0)$ onto $F_c^{-1}(t)$, and scaling surface measure and the gradient gives
  \[
    A_\vartheta(t)
    =
    r_t^{M-2}
    \int_{F_c^{-1}(0)}
    \vartheta(r_tx_{\le K})
    \fr{\varphi_M(r_tx)}{\|\nabla F_c(x)\|}
    \,\de\cH^{M-1}(x).
  \]
  The zero fiber is compact, and $\|\nabla F_c\|$ is bounded away from zero on it.
  Continuity therefore follows from dominated convergence.

  By \eqref{e:finite-coarea-pin},
  \beq\label{e:finite-pin-level-integral}
    \int\vartheta(z_{\le K})\,\nu_c(\de z)
    =\fr{A_\vartheta(0)}{p_c(0)}.
  \eeq
  The coarea formula and the approximate-identity property give
  \[
    \bbE\bigl[\vartheta(g_{\le K})\eta_\eps(F_c(g))\bigr]
    =
    \int\eta_\eps(t)A_\vartheta(t)\,\de t
    \rightarrow
    A_\vartheta(0).
  \]

  Since the tail contains at least three positive coefficients, the modulus of the characteristic function of $T_{c,>K}$ is $O(|u|^{-3/2})$.
  Fourier inversion therefore gives the bounded continuous density $p_{c,>K}$.
  Independence of the first $K$ coordinates and the tail yields
  \[
    \begin{split}
      &\bbE[\vartheta(g_{\le K})\eta_\eps(F_c(g))]
      =
      \int_{\bbR^K}
      \vartheta(z)\varphi_K(z)
      \lt[
        \int_\bbR
        \eta_\eps(S_{c,K}(z)+r)p_{c,>K}(r)\,\de r
      \rt]
      \,\de z.
    \end{split}
  \]
  Because $\eta$ is even, the expression in brackets satisfies
  \[
    (\eta_\eps*p_{c,>K})(-S_{c,K}(z))
    \rightarrow
    p_{c,>K}(-S_{c,K}(z)),
    \qquad
    |(\eta_\eps*p_{c,>K})(-S_{c,K}(z))|
    \le
    \|p_{c,>K}\|_\infty.
  \]
  Dominated convergence gives
  \[
    \bbE[\vartheta(g_{\le K})\eta_\eps(F_c(g))]
    \rightarrow
    \int_{\bbR^K}
    \vartheta(z)\varphi_K(z)
    p_{c,>K}(-S_{c,K}(z))
    \,\de z.
  \]
  Comparing the last two limits and substituting in \eqref{e:finite-pin-level-integral} proves \eqref{e:finite-pin-block-density}.
\end{proof}

\begin{lem}\label{l:pinned-law-continuity}
  Let $c^{(n)}=(c_k^{(n)})_{k\ge1}$ be finite positive vectors with at least three active coordinates, padded by zeros, and suppose that
  $
    c^{(n)}\rightarrow a
  $ in $\ell^2$, 
  where $a\in\cA$, with $\cA$ as in Lemma~\ref{l:parameter-measurability}.
  Then
  \[
    \bbW_2(\fP_{c^{(n)}},\fP_a)\rightarrow0,
    \qquad
    V_{c^{(n)}}\rightarrow V_a.
  \]
\end{lem}

\begin{proof}
  We begin with convergence of the constraint densities.
  Write $\Phi_n$ and $\Phi$ for the characteristic functions of the unconditioned constraints associated with $c^{(n)}$ and $a$, respectively.
  Thus, with the principal square root,
  \[
    \Phi_n(u)
    =
    \prod_{k\ge1}
    e^{-iuc_k^{(n)}}(1-2iuc_k^{(n)})^{-1/2},
  \]
  where only finitely many factors are nontrivial, and $\Phi$ is the corresponding infinite product.
  For
  \[
    L_u(x)=-iux-\fr12\operatorname{Log}_{\mathrm{pr}}(1-2iux),
  \]
  direct differentiation gives
  \[
    L_u'(x)=-\fr{2u^2x}{1-2iux},
    \qquad
    |L_u'(x)|\le2u^2x,
    \qquad x\ge0.
  \]
  Consequently,
  \[
    |L_u(x)-L_u(y)|
    \le
    2|u|^2(x+y)|x-y|,
    \qquad
    x,y\ge0.
  \]
  It follows that $\Phi_n\to\Phi$ locally uniformly.
  Put
  \[
    \alpha=\fr12\min\{a_1,a_2,a_3\}>0.
  \]
  For all sufficiently large $n$, the first three coordinates of $c^{(n)}$ are at least $\alpha$, and hence
  \[
    \max\{|\Phi_n(u)|,|\Phi(u)|\}
    \le
    (1+4\alpha^2u^2)^{-3/4}.
  \]
  The right side is integrable, so Fourier inversion and dominated convergence give
  \beq\label{e:constraint-density-convergence}
    \|\Phi_n-\Phi\|_{L^1(\bbR)}\rightarrow0,
    \qquad
    p_{c^{(n)}}(0)\rightarrow p_a(0)>0.
  \eeq

  We next prove convergence of every fixed pinned block.
  Fix $K\ge1$.
  Since $c_{K+j}^{(n)}\to a_{K+j}>0$ for $j=1,2,3$, the active dimension of $c^{(n)}$ is at least $K+3$ for all sufficiently large $n$.
  The preceding characteristic-function argument, applied after deleting the first $K$ coordinates, gives
  \beq\label{e:tail-density-uniform-convergence}
    \|p_{c^{(n)},>K}-p_{a,>K}\|_\infty\rightarrow0.
  \eeq
  Lemma~\ref{l:finite-pin-block-marginals} shows that the first-$K$ marginal of the finite pin has density
  \[
    h_{n,K}(z)
    =
    \fr{\varphi_K(z)}{p_{c^{(n)}}(0)}
    p_{c^{(n)},>K}\lt(
      -\sum_{k=1}^Kc_k^{(n)}(z_k^2-1)
    \rt).
  \]
  The corresponding marginal of $\nu_a$ has density
  \[
    h_K(z)
    =
    \fr{\varphi_K(z)}{p_a(0)}
    p_{a,>K}\lt(
      -\sum_{k=1}^Ka_k(z_k^2-1)
    \rt)
  \]
  by \eqref{e:canonical-pin-marginal-s}.
  Equations \eqref{e:constraint-density-convergence} and \eqref{e:tail-density-uniform-convergence}, together with continuity of $p_{a,>K}$, imply $h_{n,K}(z)\to h_K(z)$ for every $z$.
  Both sides are probability densities, so Scheffé's lemma gives
  \beq\label{e:pinned-block-TV}
    \|h_{n,K}-h_K\|_{L^1(\bbR^K)}\rightarrow0.
  \eeq

  We also require moment bounds that are uniform in the coordinate and in $n$.
  For each finite pin, assign each inactive coordinate an independent standard Gaussian and set $m_k(c)=\bbE_{\nu_c}[g_k^2]$ under this convention.
  The Fourier-inversion argument from Lemma~\ref{l:pin-moments} applies to the finite pins and gives
  \[
    m_k(c^{(n)})
    =
    \fr{\displaystyle\int_{\bbR}\Phi_n(u)(1-2iuc_k^{(n)})^{-1}\,\de u}
    {\displaystyle\int_{\bbR}\Phi_n(u)\,\de u}.
  \]
  The common integrable bound above and \eqref{e:constraint-density-convergence}, together with $|(1-2iuc_k^{(n)})^{-1}|\le1$, imply
  \beq\label{e:uniform-finite-pin-second-moments}
    \limsup_{n\to\infty}\sup_{k\ge1}m_k(c^{(n)})<\infty.
  \eeq
  The same formula, the common integrable bound above, and dominated convergence give, for every fixed $k$,
  \beq\label{e:fixed-pin-moment-convergence}
    m_k(c^{(n)})\rightarrow m_k(a).
  \eeq

  Let
  \[
    Q_{n,K}
    =
    \sum_{k=1}^Kc_k^{(n)}g_k^{(1)}g_k^{(2)},
    \qquad
    Q_K
    =
    \sum_{k=1}^Ka_kg_k^{(1)}g_k^{(2)},
  \]
  where the two replicas in each expression have the corresponding first-$K$ pinned marginal.
  Equation \eqref{e:pinned-block-TV} implies total-variation convergence of the two-replica marginals.
  Since the coefficients converge coordinatewise, $Q_{n,K}\Rightarrow Q_K$.
  Coordinatewise sign symmetry and replica independence give
  \[
    \bbE Q_{n,K}^2
    =
    \sum_{k=1}^K(c_k^{(n)})^2m_k(c^{(n)})^2
    \rightarrow
    \sum_{k=1}^Ka_k^2m_k(a)^2
    =
    \bbE Q_K^2
  \]
  by \eqref{e:fixed-pin-moment-convergence}.
  The weak-plus-second-moment characterization of Wasserstein convergence (see, e.g., \cite[Theorem 6.9]{Villani2009}) yields for every fixed $K$,
  \beq\label{e:truncated-pin-W2-convergence}
    \bbW_2\bigl(\Law(Q_{n,K}),\Law(Q_K)\bigr)\rightarrow0.
  \eeq

  It remains to remove the truncation.
  Let $Q_n$ and $Q$ denote the full finite and infinite bilinear forms.
  On the natural coupling with their truncations, coordinatewise sign symmetry gives
  \[
    \bbE|Q_n-Q_{n,K}|^2
    =
    \sum_{k>K}(c_k^{(n)})^2m_k(c^{(n)})^2,
    \qquad
    \bbE|Q-Q_K|^2
    =
    \sum_{k>K}a_k^2m_k(a)^2.
  \]
    The natural couplings and the triangle inequality give
  \[
    \begin{split}
      \bbW_2(\fP_{c^{(n)}},\fP_a)
      &\le \bbW_2\bigl(\Law(Q_{n,K}),\Law(Q_K)\bigr)+
      \lt(\sum_{k>K}(c_k^{(n)})^2m_k(c^{(n)})^2\rt)^{1/2}+
      \lt(\sum_{k>K}a_k^2m_k(a)^2\rt)^{1/2}\\
      &\stackrel{\eqref{e:uniform-finite-pin-second-moments}}{\le} \bbW_2\bigl(\Law(Q_{n,K}),\Law(Q_K)\bigr)+O(1)\cdot \lt(\sum_{k>K}(c_k^{(n)})^2\rt)^{1/2}+O(1)\cdot
      \lt(\sum_{k>K}a_k^2\rt)^{1/2}\,.
    \end{split}
  \]
Moreover, we can bound
  \[
    \sum_{k>K}(c_k^{(n)})^2
    \le
    2\|c^{(n)}-a\|_2^2
    +2\sum_{k>K}a_k^2.
  \]
  Letting first $n\to\infty$ in \eqref{e:truncated-pin-W2-convergence} and then $K\to\infty$ proves
  $
    \bbW_2(\fP_{c^{(n)}},\fP_a)\rightarrow0.
  $
  Convergence of second moments is part of $\bbW_2$ convergence, and these second moments are $V_{c^{(n)}}$ and $V_a$.
\end{proof}

\begin{proof}[Proof of Theorem~\ref{t:main}\ref{i:main-sphere}]
  Let
  \[
    \mathsf M_N^\sph
    =
    \la\delta(N^{1/3}R_{1,2})\ra^\sph.
  \]
  Lemma~\ref{l:finite-spherical-pin} gives, conditionally on $\bW$,
  \beq\label{e:spherical-finite-pin-identity}
    \mathsf M_N^\sph=\fP_{a_N}.
  \eeq
  The finite pinned-law extension is Borel on $\ell^2$ by Lemma~\ref{l:finite-pin-measurability}, and the limiting pinned-law map is Borel on $\cA$ by Lemma~\ref{l:parameter-measurability}.

  Proposition~\ref{p:augmented-edge-convergence} and Skorokhod's representation theorem give copies $\widetilde a_N$ and $\widetilde a$ of $a_N$ and $a(\chi)$ such that $\widetilde a_N\to\widetilde a$ in $\ell^2$ almost surely.
  Set $\widetilde{\mathsf M}_N=\fP_{\widetilde a_N}^{\mathrm{fin}}$, and set $\widetilde{\mathsf M}=\fP_{\widetilde a}$ on $\{\widetilde a\in\cA\}$ and $\widetilde{\mathsf M}=\delta_0$ on its null complement.
  These are measurable random elements of $(\cP_2(\bbR),\bbW_2)$.
  On a common probability-one event, every $\widetilde a_N$ is a finite positive vector padded by zeros, $\widetilde a\in\cA$, and the displayed convergence holds.
  Lemma~\ref{l:pinned-law-continuity} applies pathwise, and $\fP_{\widetilde a_N}^{\mathrm{fin}}=\fP_{\widetilde a_N}$ almost surely, so
  \[
    \bbW_2\bigl(\widetilde{\mathsf M}_N,\widetilde{\mathsf M}\bigr)
    \rightarrow0
    \qquad\text{almost surely}.
  \]
  Returning to the original laws and using \eqref{e:spherical-finite-pin-identity}, we conclude that
  \beq\label{e:spherical-outer-weak-convergence}
    \Law(\mathsf M_N^\sph)
    \Rightarrow
    \Law(\fP_{a(\chi)})\qquad\text{on }(\cP_2(\bbR),\bbW_2)\,.
  \eeq
  Proposition~\ref{p:limit-distribution-square-integrable-random} shows that the limiting law belongs to $\cP_2(\cP_2(\bbR))$.
  Lemma~\ref{l:relative-compactness} says that the sequence on the left of \eqref{e:spherical-outer-weak-convergence} is relatively compact for $\cW_2$.
  Every subsequence therefore has a further subsequence converging in $\cW_2$.
  Such convergence implies weak convergence on $(\cP_2(\bbR),\bbW_2)$, so \eqref{e:spherical-outer-weak-convergence} forces every resulting limit to equal $\Law(\fP_{a(\chi)})$.
  The subsequence principle gives
  \[
    \Law(\mathsf M_N^\sph)
    \stackrel{\cW_2}{\longrightarrow}
    \Law(\fP_{a(\chi)}).\qedhere
  \]
\end{proof}

\subsection{Deduction of the Ising overlap limit}
\label{ss:ising-overlap-limit}

\begin{proof}[Proof of Theorem~\ref{t:main}\ref{i:main-cube}]
  By the triangle inequality,
  \balnn
      \cW_2\lt(
        \Law\lt(\la\delta(N^{1/3}R_{1,2})\ra\rt),
        \Law(\fP_{a(\chi)})
      \rt)
      &\le
      \cW_2\lt(
        \Law\lt(\la\delta(N^{1/3}R_{1,2})\ra\rt),
        \Law\lt(\la\delta(N^{1/3}R_{1,2})\ra^\sph\rt)
      \rt) \\
      &\qquad+
      \cW_2\lt(
        \Law\lt(\la\delta(N^{1/3}R_{1,2})\ra^\sph\rt),
        \Law(\fP_{a(\chi)})
      \rt).
  \ealnn
  The first term tends to zero by Theorem~\ref{t:sphere-to-cube}, and the second tends to zero by Theorem~\ref{t:main}\ref{i:main-sphere}.
\end{proof}

\bibliographystyle{alpha}
\bibliography{bib}

\end{document}